%% file: DD-Laguerre.tex
\documentclass[12pt,a4paper]{report} 
\usepackage{amsfonts,amssymb,amsmath,amsthm,graphicx}
\usepackage{bm}      
\usepackage[english]{babel}
\usepackage{makeidx}

  {\theoremstyle{definition}
  \newtheorem{defi}{Definition}[section]
  
  \newtheorem{exa}[defi]{Example}
  \newtheorem{exas}[defi]{Examples}
  \newtheorem{exer}[defi]{Exercise}
  \swapnumbers
  \newtheorem{nrtxt}[defi]{}
  }

  {\theoremstyle{plain} 
  \newtheorem{lem}[defi]{Lemma}
  
  \newtheorem{thm}[defi]{Theorem}
  \newtheorem{cor}[defi]{Corollary}  }
  
\newenvironment{txt}{\begin{trivlist}\item}{\end{trivlist}}

\makeindex

\newcommand{\noslant}[1]{{\rm#1}}

\newcommand{\eps}{{\varepsilon}}

\newcommand{\bC}{{\mathbb C}}

\newcommand{\bH}{{\mathbb H}}

\newcommand{\bP}{{\mathbb P}}

\newcommand{\bR}{{\mathbb R}}

\newcommand{\bZ}{{\mathbb Z}}

\newcommand{\cB}{{\mathcal B}}
\newcommand{\cC}{{\mathcal C}}
\newcommand{\cD}{{\mathcal D}}

\newcommand{\cR}{{\mathcal R}}
\newcommand{\cS}{{\mathcal S}}

\newcommand{\vb}{{\bm{b}}}

\newcommand{\vv}{{\bm{v}}}
\newcommand{\vw}{{\bm{w}}}
\newcommand{\vx}{{\bm{x}}}
\newcommand{\vy}{{\bm{y}}}

\newcommand{\vC}{{\bm{C}}}

\newcommand{\vM}{{\bm{M}}}

\newcommand{\vV}{{\bm{V}}}

\DeclareMathOperator{\rad}{rad}
\DeclareMathOperator{\GL}{GL}
\DeclareMathOperator{\Z}{Z}
\DeclareMathOperator{\Aut}{Aut}

\DeclareMathOperator{\diag}{diag}
\DeclareMathOperator{\GF}{GF}
\DeclareMathOperator{\id}{id}
\DeclareMathOperator{\PGL}{PGL}
\DeclareMathOperator{\T}{T}
\newcommand{\dis}{\mathbin{\scriptstyle\triangle}}
\newcommand{\notdis}{\mathbin{\not\scriptstyle\triangle}}

\newcommand{\rel}{\mathrel{\cR}}
\newcommand{\card}{\mathbin{\#}}
\newcommand{\Matrixfeld}[4]{\left#1\!\begin{array}{*{#3}{c}}#4\end{array}\!\right#2}
\newcommand{\SMatrixfeld}[4]{\hbox{\scriptsize\arraycolsep=.5\arraycolsep
  $\left#1\!\begin{array}{*{#3}{c}}#4\end{array}\!\right#2$}}
\newcommand{\RMatrixfeld}[4]{\left#1\!\begin{array}{*{#3}{r}}#4\end{array}\!\right#2}

\newcommand{\Mat}{\Matrixfeld()}
\newcommand{\SMat}{\SMatrixfeld()}
\newcommand{\RMat}{\RMatrixfeld()}

\newcommand{\tsk}{$t$-$(s,k,\lambda_t)$}

\begin{document}
\include{DD-Laguerre0}
\include{DD-Laguerre1}
\include{DD-Laguerre2}
\include{DD-Laguerre3}
\include{DD-Laguerre4}
\include{DD-Laguerre5}
\small
\addcontentsline{toc}{chapter}{Bibliography} 

\input{DD-Laguerre-lit}
\footnotesize
\addcontentsline{toc}{chapter}{Index} 
\printindex 
\end{document}

%% file: DD-Laguerre0.tex
\author{Hans Havlicek\\
Institut f\"{u}r Diskrete Mathematik und Geometrie\\
Technische Universit\"{a}t Wien\\
Wiedner Hauptstra{\ss}e 8--10\\
A-1040 Wien\\
Austria \\
\texttt{havlicek@geometrie.tuwien.ac.at}}

\title{Divisible Designs, Laguerre Geometry, and Beyond}

\maketitle

\tableofcontents


%% file: DD-Laguerre1.tex
\chapter{Introduction}

\begin{txt}

This is a revised and updated version of our lectures notes \cite{havl-06a}
from the \emph{Summer School on Combinatorial Geometry and Optimisation 2004
``Giuseppe Tallini''} which took place at the \emph{Catholic University of
Brescia}, Italy.

In these notes we aim at bringing together design theory and projective
geometry over a ring. Both disciplines are well established, but the results on
the interaction between them seem to be rare and scattered over the literature.
Thus our main goal is to present the basics from either side, to develop, or at
least sketch, the principal connections between them, and to make
recommendations for further reading. There is no attempt to provide
encyclopedic coverage with expansive notes and references.

In Chapter \ref{chap:DD} we start from the scratch with divisible designs.
Loosely speaking, a divisible design is a finite set of points which is endowed
with an equivalence relation and a family of distinguished subsets, called
blocks, such that no two distinct points of a block are equivalent.
Furthermore, there have to be several constants, called the parameters of the
divisible design, as they govern the basic combinatorial properties of such a
structure. Our exposition includes a lot of simple examples. Also, we collect
some facts about group actions. This leads us to a general construction
principle for divisible designs, due to \textsc{Spera}. This will be our main
tool in the subsequent chapters.

Next, in Chapter \ref{chap:P(R)} we take a big step by looking at the classical
Laguerre geometry over the reals. This part of the text is intended mainly as a
motivation and an invitation for further reading. Then we introduce our
essential geometric concept, the projective line over a ring. Although we shall
be interested in finite rings only, we do not exclude the infinite case. In
fact, a restriction to finite rings would hardly simplify our exposition. From
a ring containing a field, as a subring, we obtain a chain geometry. Again, we
take a very short look at some classical examples, like M\"{o}bius geometries. Up
to this point the connections with divisible designs may seem vague. However,
if we restrict ourselves to finite local rings then all the prerequisites
needed for constructing a divisible design are suddenly available, due to the
presence of a unique maximal ideal in a local ring.

Chapter \ref{chap:DD.GL} is entirely devoted to the construction of a divisible
design from the projective line over a finite local ring. The particular case
of a local algebra is discussed in detail, but little seems to be known about
the case of an arbitrary finite local ring, even though such rings are
ubiquitous. It is worth noting that the isomorphisms between certain divisible
design can be described in terms of Jordan isomorphisms of rings and
projectivities; strictly speaking this applies to divisible designs which stem
from chain geometries over local algebras with sufficiently large ground
fields. Geometric mappings arising from Jordan homomorphisms are rather
involved, and the related proofs have the tendency to be very technical; we
therefore present this material without giving a proof.

Chapter \ref{chap:FCG} can be considered as an outlook combined with an
invitation for further research. We sketch how one can obtain an equivalence
relation on the projective line over any ring via the Jacobson radical of the
ring. Recall that such an equivalence relation is one of the ingredients for a
divisible design. The maximal ideal of a local ring is its Jacobson radical, so
that we can generalise some of our results from a local to an arbitrary ring.
It remains open, however, if this equivalence relation could be used to
construct successfully a divisible design even when the ring is not local.
Finally, we collect some facts about finite chain geometries. Their
combinatorial properties are---in a certain sense---almost those of divisible
designs, but no systematic treatment seems to be known.
\end{txt}


%% file: DD-Laguerre2.tex
\chapter{Divisible Designs}\label{chap:DD}

\section{Basic concepts and first examples}

\begin{nrtxt}
Suppose that a tournament is to take place with $v$ participants coming from
various teams, each team having the same number of members, say $s$. In order
to avoid trivialities, we assume $v>0$ and $s>0$. So there are $v/s$ teams. The
tournament consists of a number of games. In any game $k\geq 2$ participants
from different teams play against each other. Of course, there should be at
least two teams, i.~e., $2\leq v/s$.

The problem is to organise this tournament in such a way that all participants
are ``treated equally''. Strictly speaking, the objective is as follows:

\emph{The number of games in which any two members from different teams play
against each other has to be a constant value, say $\lambda_2$.}

In this way it is impossible that one participant would have the advantage of
playing over and over again against a small number of members from other teams,
whereas others would face many different counterparts during the games.

In the terminology to be introduced below, this problem amounts to constructing
a $2$-$(s,k,\lambda_2)$-divisible design with $v$ elements. The points of the
divisible design are the participants, the point classes are the teams, and the
blocks correspond to the games. Many of our examples will give solutions to
this problem for certain values of $s$, $k$, $\lambda_2$, and $v$.
\end{nrtxt}

\begin{nrtxt}
Throughout this chapter we adopt the following assumptions: $X$ is a finite set
with an equivalence relation ${\rel}\subset X\times X$. We denote by $[x]$ the
$\rel$-equivalence class of $x \in X$ and define
\begin{equation}\label{eq:def-S}
   \cS:=\{[x]\mid x \in X\}.
\end{equation}
A subset $Y$ of $X$ is called
\emph{$\rel$-transversal}\index{R-transversal@$\rel$-transversal
subset}\index{subset R-transversal@subset!$\rel$-transversal} if $\card (Y \cap
[x]) \le 1$ for all $x \in X$. Observe that here the word ``transversal''
appears in a rather unusual context, since it is not demanded that $Y$ meets
\emph{all\/} equivalence classes in precisely one element. Cf., however, the
definition of a transversal divisible design in \ref{:transversal.DD}.
\end{nrtxt}

\begin{defi}\label{def:DD}
A triple $\mathcal{D}=(X,\cB,\cS)$ is called a \tsk-\emph{divisible
design}\index{design!divisible} if there exist positive integers
$t,s,k,\lambda_t$ such that the following axioms hold:

\begin{itemize}

\item[(A)] $\mathcal{B}$ is a set of $\rel$-transversal subsets of $X$ with
$\card B=k$ for all  $B \in\mathcal{B}$.

\item[(B)] $\card [x] =s$ for all $x \in X$.

\item[(C)] For each $\rel$-transversal $t$-subset $Y\subset X$ there exist
exactly $\lambda_t$ elements of $\mathcal{B}$ containing $Y$.

\item[(D)] $t\leq \frac{v}{s}$, where $v:=\card  X$.
\end{itemize}
The elements of $X$ are called \emph{points}\index{point}, those of
$\mathcal{B}$ \emph{blocks}\index{block}, and the elements of $\cS$ \emph{point
classes}\index{point class}.
\end{defi}

\begin{txt}
We shall frequently use the shorthand ``DD''\index{DD} for ``divisible
design''. Sometimes we shall speak of a $t$-DD without explicitly mentioning
the remaining \emph{parameters}\index{parameters of a DD}\index{DD!parameters
of a} $s$, $k$, and $\lambda_t$. According to our definition, a block is merely
a subset of $X$. Hence the DDs which we are going to discuss are
\emph{simple}\index{DD!simple}\index{simple DD}, i.~e., we do not take into
account the possibility of ``repeated blocks''. Cf.\ \cite[p.~2]{beth+j+l-99a}
for that concept.

Since $\cS$ is determined by $\cR$ and vice versa, we shall sometimes also
write a divisible design in the form $(X,\cB,\rel)$ rather than $(X,\cB,\cS)$.
\end{txt}

\begin{nrtxt}\label{:basis.DD}
Let us write down some basic properties of a \tsk-DD. Since $s,t\geq 1$, axiom
(D) implies that
\begin{equation}\label{eq:v>0}
    \card X=v\geq st \geq 1
\end{equation}
or, said differently, that $X\neq \emptyset$. From this and (B) we infer that
\begin{equation}\label{eq:card-S}
  \card\cS=\frac vs\geq 1.
\end{equation}
Hence, by (D) and (B), there exists at least one $\rel$-transversal $t$-subset
of $X$, say $Y_0$. By virtue of (C), this $Y_0$ is contained in $\lambda_t\geq
1$ blocks so that
\begin{equation}\label{eq:card-B}
    \card\cB=:b\geq 1.
\end{equation}
So, since $\cB\neq\emptyset$, we can derive from axiom (A) and
(\ref{eq:card-S}) the inequality
\begin{equation}\label{eq:card-Block}
 \card B = k\leq \frac v s\, \mbox{ for all }B\in\cB.
\end{equation}
\end{nrtxt}

\begin{nrtxt}\label{:transversal.DD}
A divisible design is called
\emph{transversal}\index{DD!transversal}\index{transversal DD} if each block
meets all point classes, otherwise it is called
\emph{regular}\index{DD!regular}\index{regular DD}. Hence a \tsk-DD is
transversal if, and only if equality holds in (\ref{eq:card-Block}).

During the last decades there has been a change of terminology. Originally, the
point classes of a DD were called \emph{point groups\/}\index{point group} and
DDs carried the name \emph{group-divisible designs}\index{group-divisible
design}\index{design!group-divisible}. In order to avoid confusion with the
algebraic term ``group'', in \cite{beth+j+l-85} this name was changed to read
\emph{groop-divisible designs}\index{groop-divisible
design}\index{design!groop-divisible}. We shall not use any of these phrases.
\end{nrtxt}

\begin{nrtxt} Let us add in passing that some authors use slightly different
axioms for a DD in order to exclude certain cases that do not deserve
interest. For example, according to our definition $s=v$ is allowed, but this
forces $t=k=1$.

On the other hand, our axiom (D) is essential in order to rule out trivial
cases which would cause a lot of trouble. If we would allow $t>\frac vs$ then
there would not be any $\rel$-transversal $t$-subset of $X$, and (C) would hold
in a trivial manner. Such a value for $t$ would therefore have no meaning at
all for a structure $\cD=(X,\cB,\cS)$.
\end{nrtxt}

\begin{exas}\label{bsp:erste}
We present some examples of DDs.
\begin{enumerate}

\item\label{bsp:erste.pappos}

We consider the \emph{Pappos configuration}\index{Pappos configuration} in the
real projective plane which is formed by $9$ points and $9$ lines according to
Figure \ref{abb:pappos}.
\begin{figure}[h]\begin{center}\unitlength1cm\small
\begin{picture}(5,4.5)
 \put(0,0.5){\includegraphics[width=5cm]{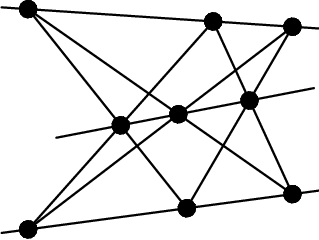}}
 \put(0.3,0.15){$p_1$}
 \put(2.8,0.5){$q_1$}
 \put(4.5,0.7){$r_1$}

 \put(4.15,2.4){$p_2$}
 \put(2.65,2.78){$q_2$}
 \put(1.3,2.4){$r_2$}

 \put(0.3,4.4){$p_3$}
 \put(3.2,4.2){$q_3$}
 \put(4.5,4.1){$r_3$}
\end{picture}\caption{\label{abb:pappos}Pappos configuration}
\end{center}\end{figure}
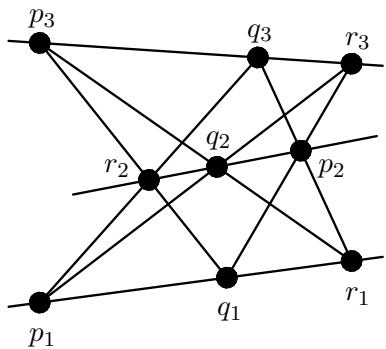
We obtain a $2$-$(3,3,1)$-DD, say $\cD$, as follows: Let
\begin{equation*}
  X:=\{p_1,p_2,p_3, q_1,q_2,q_3,r_1,r_2,r_3\},
\end{equation*}
i.~e., $v=9$. The blocks are, by definition, the $3$-subsets of collinear
points in $X$, so that $k=3$. We define three point classes, namely
$\{p_1,p_2,p_3\}$, $\{q_1,q_2,q_3\}$, and $\{r_1,r_2,r_3\}$, each with
$s=3$ elements. Then for any two points from distinct point classes are is
a unique block containing them. So $t=2$ and $\lambda_2=1$. This DD is
transversal.

\item\label{bsp:erste.oktaeder}

Let us take a \emph{regular octahedron}\index{octahedron!regular}\index{regular
octahedron} in the Euclidean $3$-space (Figure \ref{abb:oktaeder}), and let us
turn it into a DD as follows:
\begin{figure}[h]\begin{center}\unitlength1cm
\begin{picture}(5,3.5)
 \put(0,0){\includegraphics[width=5cm]{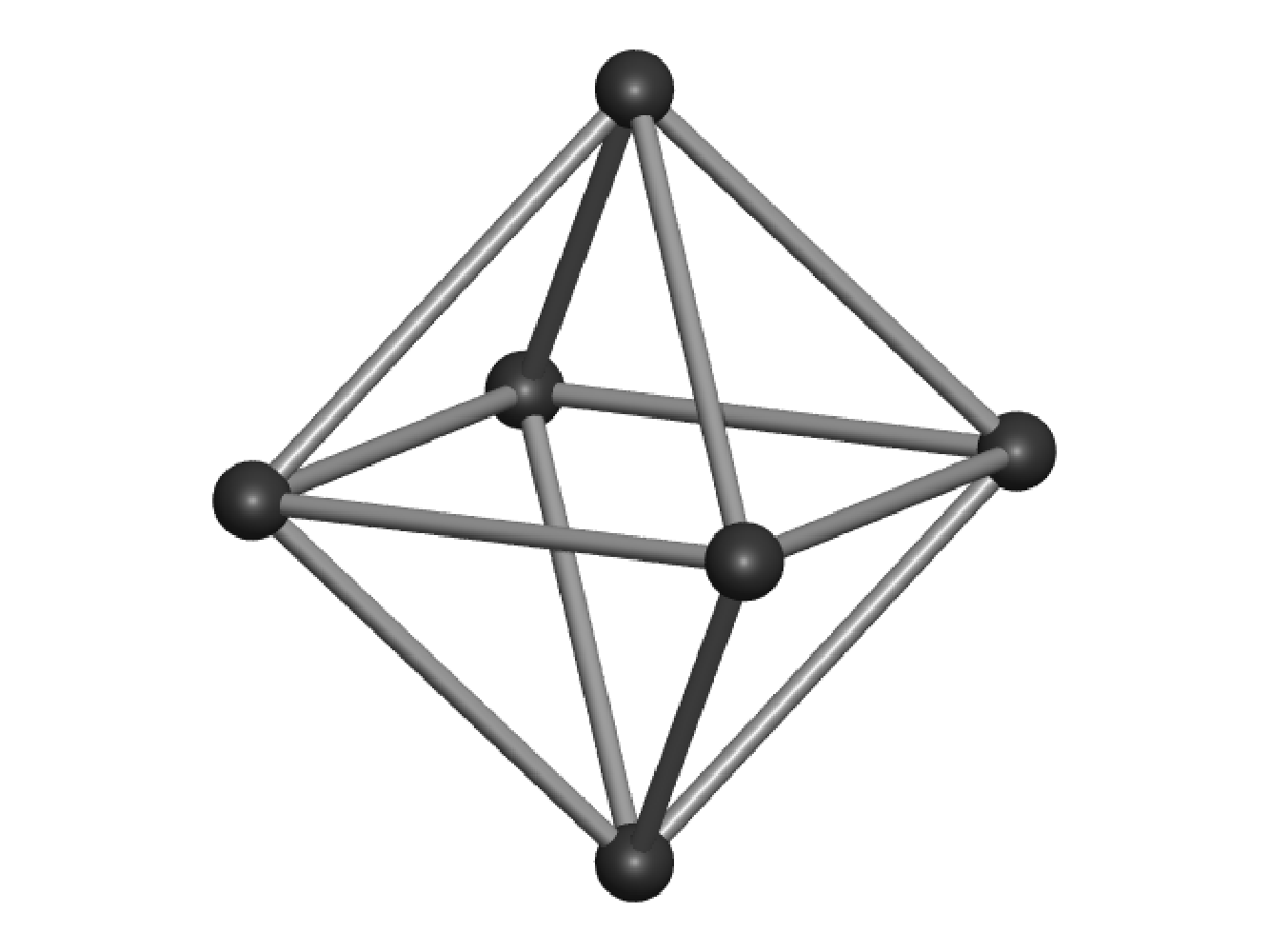}}
\end{picture}
\caption{\label{abb:oktaeder}Octahedron}
\end{center}\end{figure}
Denote by $X$ the set of all $v=6$ vertices of the octahedron. For all $p,q\in
X$ we put $p\rel q$ if, and only if, $p$ and $q$ are opposite vertices. Hence
$s=2$. The blocks are defined as the triangular faces, whence $k=3$. So we get
a transversal $3$-$(2,3,1)$-divisible design.

\item\label{bsp:erste.PG23}

Our next example is the \emph{projective
plane}\index{plane!projective}\index{projective plane} of order three which is
depicted on the left hand side of Figure \ref{abb:PG23}. It is a
$2$-$(1,4,1)$-DD with $v=13$ points. There are $13$ blocks; they are given by
those subsets of the point set which consist of $k=4$ points on a common curve.
(Some of these curves are segments, others are not.) There are $13$ point
classes, because $s=1$ means that all point classes are singletons.
\par
We shall not need the definition of a finite projective plane and refer to
\cite[p.~6]{beth+j+l-99a}. Let us add, however, that in the theory of
projective planes one speaks of \emph{lines\/}\index{line} rather than blocks.
The \emph{order\/}\index{projective plane!order of a}\index{order!of a
projective plane} of a projective plane is defined to be $k-1$ if there are $k$
points on one (or, equivalently, on every) line.

Let us remove one point from the point set of this projective plane. Also, let
us redefine the point classes as the four truncated lines (illustrated by thick
segments and a thick circular arc), the other nine lines remain as blocks. This
yields a $2$-$(3,4,1)$-DD.

If we delete one line and all its points from the projective plane of order
three then we obtain the \emph{affine plane}\index{affine
plane}\index{plane!affine} of order three. Each of the twelve remaining lines
gives rise to a block with three points, the point classes are defined as
singletons. As before, one speaks of (affine) lines rather than blocks in the
context of affine planes. Observe that the \emph{order}\index{affine
plane!order of an}\index{order!of an affine plane} of an affine plane is just
the number of points on one (or, equivalently, on every) line. See
\cite[p.~8]{beth+j+l-99a} for further details.

This affine plane is a $2$-$(1,3,1)$-DD with $v=9$ points and, as before, all
point classes are singletons. See the third picture in Figure \ref{abb:PG23}.
Two lines of an affine plane are called
\emph{parallel}\index{lines!parallel}\index{parallel lines} if they are
identical or if they have no point in common.

Finally, we change the set of lines and the set of point classes of this affine
plane as follows: We exclude three mutually parallel lines from the line set,
turn them into point classes, and disregard the one-element point classes of
the underlying affine plane. The remaining nine lines are considered as blocks.
In this way a $2$-$(3,3,1)$-DD with $v=9$ points is obtained. On the right hand
side of Figure \ref{abb:PG23} the bold vertical segments represent the point
classes.
\begin{figure}[h]\begin{center}\unitlength1cm
\begin{picture}(11,3.2)
 \put(0.5,0){\includegraphics[width=11cm]{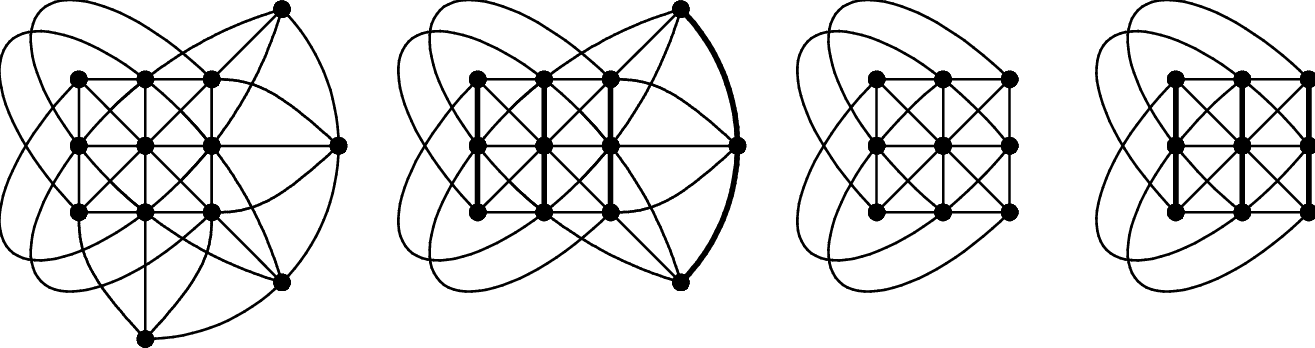}}
\end{picture}
\caption{\label{abb:PG23}DDs from the projective plane of order $3$}
\end{center}\end{figure}

\item\label{bsp:erste.PG22}

We proceed as in the previous example, but starting with the projective plane
of order two which is a $2$-$(1,3,1)$-DD with $v=7$ points. In this way we
obtain a $2$-$(2,3,1)$-DD with $v=6$ points, a $2$-$(1,2,1)$-DD with $v=4$
points (the affine plane of order $2$), and a $2$-$(2,2,1)$-DD with $v=4$
points. See Figure \ref{abb:PG22}.
\begin{figure}[h]\begin{center}\unitlength1cm
\begin{picture}(11,2.75)
 \put(0.5,0){\includegraphics[width=11cm]{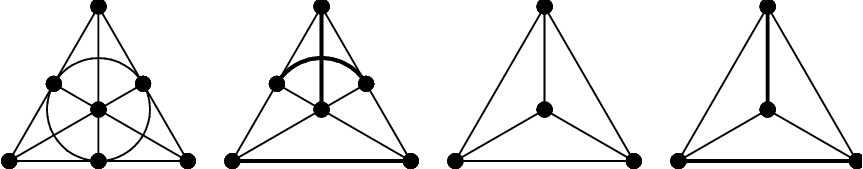}}
\end{picture}
\caption{\label{abb:PG22} DDs from the projective plane of order $2$}
\end{center}\end{figure}
\end{enumerate}
\end{exas}

\begin{txt}
It is easy to check that the $3$-DD from Example (\ref{bsp:erste.oktaeder}) is
also a $2$-DD; likewise all our $2$-DDs are at the same time $1$-DDs. Thus the
previous examples illustrate the following result:
\end{txt}

\begin{thm}\label{thm:i-DD}
Let $\cD$ be a \tsk-DD with $t\ge 2$ and let $i$ be an integer such that $1\leq
i\leq t$. Then $\cD$ is also an $i$-$(s,k,\lambda_i)$-DD with
\begin{equation}\label{eq:lambda_i}
    \lambda_i=\lambda_t\,\frac{\Mat2{vs^{-1} -i\\t-i} s^{t-i}}{\Mat2{k-i\\t-i}}\,.
\end{equation}
\end{thm}

\begin{proof} 
We fix one transversal $i$-subset $I$. The proof will be accomplished by
counting in two ways the number of pairs $(Y,B)$, where $Y$ is a
$(t-i)$-subset of $X$ such that $I\cup Y$ is a transversal $t$-subset, and
where $B$ is a block containing $I\cup Y$.

On the one hand, let us single out one of the $\lambda_i$ blocks containing
$I$. Then there are
\begin{equation*}
  \Mat2{k-i\\t-i}
\end{equation*}
possibilities to choose a $Y$ within that particular block.

On the other hand, to select an arbitrary $Y$ amounts to the following: First
choose $t-i$ point classes out of the $vs^{-1}-i$ point classes that are
disjoint from $I$ \big(cf.\ (\ref{eq:card-S})\big), and then choose in each of
these point classes a single point (out of $s$). Hence there are precisely
\begin{equation*}
  \Mat2{vs^{-1} -i\\t-i}s^{t-i}
\end{equation*}
ways to find such a $Y$. For every $Y$ there are $\lambda_t$ pairs $(Y,B)$
with the required property.

Altogether we obtain
\begin{equation}\label{eq:doppelt}
  \lambda_{i}\Mat2{k-i\\t-i}=
  \lambda_t\Mat2{vs^{-1} -i\\t-i}s^{t-i}
\end{equation}
which completes the proof.
\end{proof}

\begin{nrtxt}
Theorem \ref{thm:i-DD} enables us to calculate several other
\emph{parameters}\index{parameters of a DD}\index{DD!parameters of a} of a
\tsk-DD. Letting $i=0$ in formula (\ref{eq:lambda_i}) provides the number of
blocks, i.~e.\
\begin{equation}\label{eq:blockanzahl}
      b:=\card \cB=\lambda_t\,\frac{\Mat2{vs^{-1}\\t} s^{t}}{\Mat2{k\\t}}\,.
\end{equation}
Likewise, for $i=1$ we obtain the number
\begin{equation}\label{eq:def-r}
    r:=\lambda_1
\end{equation}
of blocks through a point which is therefore a constant. Provided that $i=t-1$
formula (\ref{eq:lambda_i}) reads
\begin{equation}\label{eq:lambda_t-1}
  \lambda_{t-1}=\lambda_t\,\frac{v-st+s}{k- t+1}\,.
\end{equation}
By Theorem \ref{thm:i-DD}, formula (\ref{eq:lambda_t-1}) remains valid if $t$
is replaced with an integer $t'$, subject to the condition $1\leq t'\leq t$.
Hence we infer the equation
\begin{equation}\label{eq:bk=rv}
  bk=rv
\end{equation}
by letting $t'=1$. For $t\geq 2$ we may let $t'=2$ which gives
\begin{equation}\label{eq:lambda2}
  r({k-1})=\lambda_2({v-s}).
\end{equation}
The last two equations are just particular cases of formula
(\ref{eq:doppelt}).
\end{nrtxt}

\begin{nrtxt}
A divisible design with $s=1$ is called a \emph{design}\index{design}; we refer
to \cite{colburn+d-07a}, \cite{hugh+p-88}, \cite{lindner+r-09a}, or the two
volumes \cite{beth+j+l-99a} and \cite{beth+j+l-99b}. In design theory the
parameter $s$ is not taken into account, and a $t$-$(1,k,\lambda_t)$-DD with
$v$ points is often called a $t$-$(v,k,\lambda_t)$-design. Of course, this is a
\emph{different notation\/} and we urge the reader not to draw the erroneous
conclusion ``$v=s$'' when comparing these lecture notes with a book on design
theory.

We have already met examples of designs in Examples \ref{bsp:erste}
(\ref{bsp:erste.PG23}) and (\ref{bsp:erste.PG22}), namely the projective and
affine planes of orders three and two. However, designs are not the topic of
this course. Instead, we shall focus our attention on the case when $s>1$.
\end{nrtxt}

\begin{nrtxt}
If $\cD=(X,\cB,\cS)$ is a \tsk-DD and $\cD'=(X',\cB',\cS')$ is a
$t'$-$(s',k',\lambda'_{t'})$-DD then an
\emph{isomorphism\/}\index{isomorphism!of DDs}\index{DDs!isomorphism of} is a
bijection
\begin{equation*}
    \varphi:X\to X' : p\mapsto p^\varphi
\end{equation*}
such that
\begin{eqnarray}
  B\in\cB &\Leftrightarrow& B^\varphi\in\cB'\label{eq:iso.B} \\
  S\in\cS &\Leftrightarrow& S^\varphi\in\cS'.\label{eq:iso.S}
\end{eqnarray}
Clearly, the inverse mapping of an isomorphism is again an isomorphism. If the
product of two isomorphisms is defined (as a mapping) then it is an
isomorphism. The set of all isomorphisms of a DD onto itself, i.~e.\ the set of
all \emph{automorphisms}\index{automorphism!of a DD}\index{DD!automorphism of
a}, is a group under composition of mappings.
\end{nrtxt}

\begin{nrtxt}
Suppose that there exists an isomorphism of a \tsk-DD $\cD$ onto a
$t'$-$(s',k',\lambda'_{t'})$-DD $\cD'$. Such DDs are said to be
\emph{isomorphic}. Then
\begin{equation*}
    v=v',\; s=s'\mbox{, and }k=k'.
\end{equation*}
However, in view of Theorem \ref{thm:i-DD} we may have $t\neq t'$. Thus we
impose the extra condition that the parameters $t$ and $t'$ are maximal, i.~e.,
$\cD$ is a $t$-DD but not a $(t+1)$-DD, and likewise for $\cD'$. Then, clearly,
\begin{equation*}\label{}
  t=t'\mbox{ and }\lambda_t=\lambda'_{t'}.
\end{equation*}
\end{nrtxt}

\begin{nrtxt}
Condition (\ref{eq:iso.S}) in the definition of an isomorphism can be replaced
with the seemingly weaker but nevertheless equivalent condition
\begin{equation}\label{eq:iso.S-scharf}
   S\in\cS \;\Rightarrow\; S^\varphi\in\cS'\mbox{:}
\end{equation}
Suppose that we are given a bijection $\varphi:X\to X'$ satisfying
(\ref{eq:iso.S-scharf}). If $S^\varphi\in\cS'$ for some subset $S$ of $X$ then
there is an $x\in S$. Hence $x^\varphi\in S^\varphi\cap[x]^\varphi$ with
$[x]^\varphi\in\cS'$ by (\ref{eq:iso.S-scharf}). Since two equivalence classes
with a common element are identical, we get $S^\varphi=[x]^\varphi$ and,
finally, $S=[x]\in\cS$. In sharp contrast to this result, the equivalence sign
in (\ref{eq:iso.B}) is essential. Cf.\ Example \ref{bsp:iso-gegenbsp} below.

We may even drop condition (\ref{eq:iso.S}) in the following particular
situation: Let $\varphi:X\to X'$ be a bijection of a $2$-DD $\cD$ onto a $2$-DD
$\cD'$ such that (\ref{eq:iso.B}) holds. Then, for all $x,y\in X$ with $x\neq
y$ we have $x\rel y$ if, and only if, there exists a block containing $x$ and
$y$. The same kind of characterisation applies to $\cD'$. Hence $x\rel y$ is
equivalent to $x^\varphi\rel' y^\varphi$ for all $x,y\in X$.
\end{nrtxt}

\begin{exa}\label{bsp:iso-gegenbsp}
Let us consider once more a \emph{regular
octahedron}\index{octahedron!regular}\index{regular octahedron} in the
Euclidean $3$-space. We turn the set of its vertices into a $2$-DD with $6$
points in two different ways (Figure \ref{abb:2oktaeder}): For both DDs the
point classes are the $2$-sets of opposite vertices. However, the blocks are
different. Firstly, we take \emph{all\/} $8$ triangular faces as blocks (left
image). This gives a $2$-$(2,3,2)$-DD which is also a $3$-DD. Cf.\ Example
\ref{bsp:erste} (\ref{bsp:erste.oktaeder}). Secondly, only $4$ triangular faces
(given by the shaded triangles in the right image) are considered as blocks, so
that a $2$-$(2,3,1)$-DD is obtained.
\begin{figure}[h]\begin{center}\unitlength1cm
\begin{picture}(10,3.5)
 \put(5,0){\includegraphics[width=5cm]{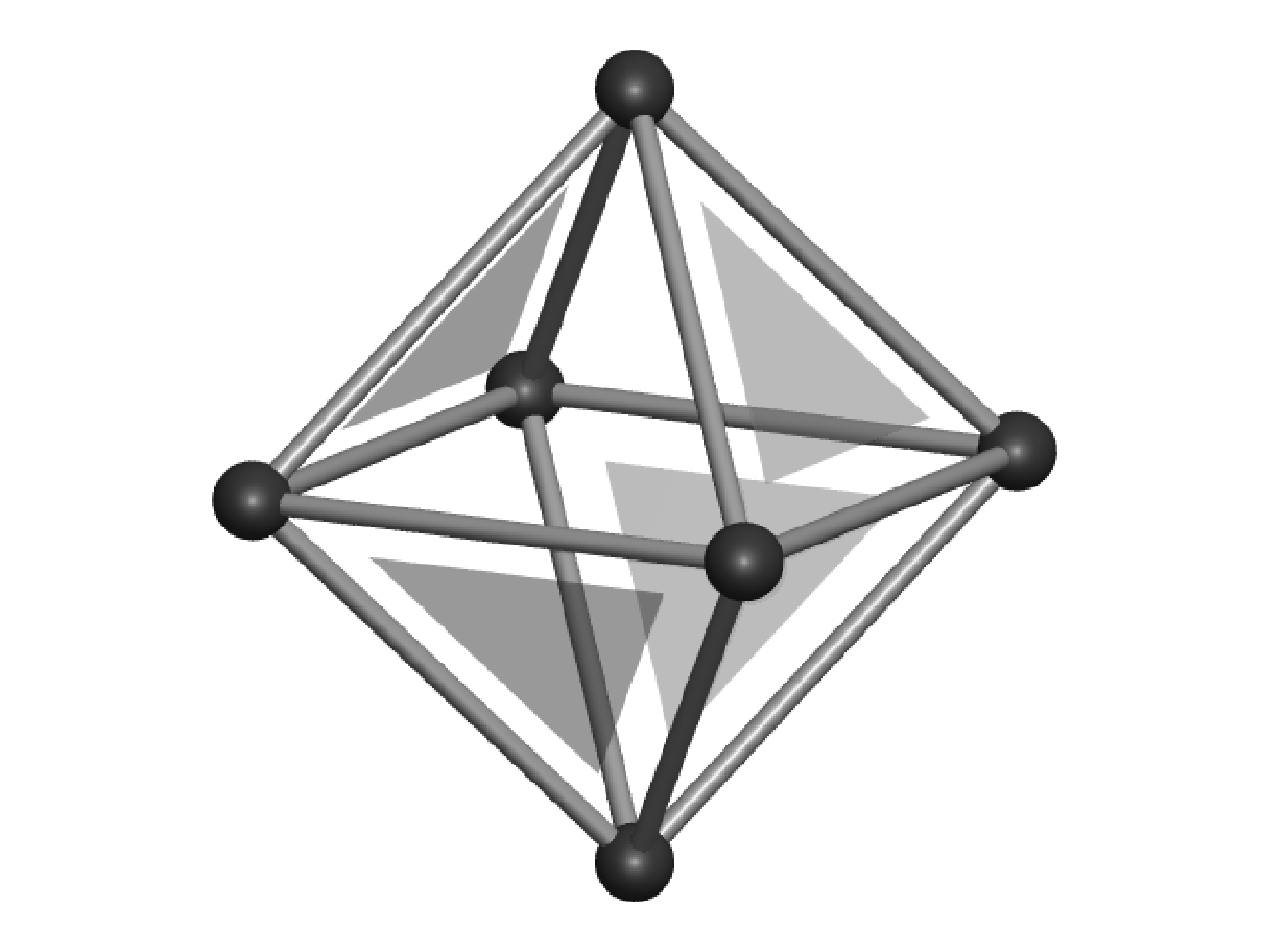}}
 \put(0,0){\includegraphics[width=5cm]{tetraeder.eps}}
\end{picture}
\caption{\label{abb:2oktaeder}Two non-isomorphic $2$-DDs from an octahedron}
\end{center}\end{figure}

Observe that the identity mapping $\id_X$ maps every block of the second
design onto a block of the first design, but not vice versa. Hence a
bijection between the point sets of DDs which preserves point classes in both
directions and blocks in one direction only, need not be an isomorphism.
\end{exa}

\begin{exer}
Which of the DDs from Examples \ref{bsp:erste} and \ref{bsp:iso-gegenbsp} are
isomorphic?
\end{exer}

\section{Group actions}

\begin{nrtxt}
Let us recall that all bijections (or \emph{permutations}\/)\index{permutation}
of a finite set\footnote{Most of the results from this section remain true for
an infinite set $X$.} $X$ form the \emph{symmetric group\/}\index{symmetric
group}\index{group!symmetric} $S_X$. If $G$ is any group then a homomorphism
\begin{equation*}
    \alpha:G\to S_X : g\mapsto g^\alpha
\end{equation*}
is called a \emph{permutation representation\/}\index{permutation
representation}\index{representation} of $G$. In this case the group $G$ is
also said to \emph{operate\/}\index{group operation} or
\emph{act\/}\index{group action} on $X$ via $\alpha$. In fact, each $g\in G$
yields the bijection
\begin{equation*}
    g^\alpha : X\to X: x\mapsto x^{(g^\alpha)}.
\end{equation*}
Whenever $\alpha$ is clear from the context, then we shall write $x^g$ for the
image of $x$ under the permutation $g^\alpha$. Thus, if the composition in $G$
is written multiplicatively, we obtain
\begin{equation*}
  x^{(gh)} =(x^g)^h \mbox{ for all }x\in X\mbox{ and all }g,h\in G.
\end{equation*}

Provided that $\alpha$ is injective the representation is called
\emph{faithful}\index{representation!faithful}\index{faithful representation}.
So for a faithful representation we have $\ker\alpha=\{1_G\}$ as is kernel, and
we can identify $G$ with its image $G^\alpha$. However, in most of our examples
the representation will not be faithful, i.~e., there will be distinct elements
of $G$ which yield the same permutation on $X$.
\end{nrtxt}

\begin{nrtxt}\label{:aktion.eigensch}
For the remaining part of this section we suppose that $G$ acts on $X$ (via
$\alpha$).

For each $x\in X$ we write $x^G:=\{x^g\mid g\in G\}$ for the
\emph{orbit}\index{orbit} of $x$ under $G$. The set of all such orbits is a
partition of $X$. If $X$ itself is an orbit then $G$ is said to operate
\emph{transitively\/}\index{group action!transitive}\index{transitive group
action} on $X$. This means that for any two elements $x,y\in X$ there is at
least one $g\in G$ with $x^g=y$. If, moreover, this $g$ is always uniquely
determined then the action of $G$ is called \emph{regular}\index{regular group
action}\index{group action!regular} or \emph{sharply transitive}\index{group
action!sharply transitive}\index{sharply transitive group action}. If $G$
operates regularly on $X$ then the representation is necessarily faithful,
since every $g\in\ker\alpha$ has the property $x^g=x$ for all $x\in X$, whence
$g=1_G$.

The given group $G$ acts also in a natural way on certain other sets which are
associated with $X$. E.g., for every non-negative integer $t$, the group $G$
acts on the $t$-fold product $X^t$ by
\begin{equation*}
  (x_1,x_2,\ldots,x_t)^g:= (x_1^g,x_2^g,\ldots,x_t^g).
\end{equation*}
If this is a transitive action on the subset of $t$-tuples with \emph{distinct
entries\/} from $X$ then one says that $G$ acts
\emph{$t$-transitively\/}\index{group
action!t-transitive@$t$-transitive}\index{t-transitive group
action@$t$-transitive group action} on $X$.

Moreover, for $t\leq \card X$, the group $G$ acts on the (non empty) set
$\SMat2{ X\\t}$ of all $t$-subsets of $X$ by
\begin{equation*}
  \{x_1,x_2,\ldots,x_t\}^g:=\{x_1^g,x_2^g,\ldots,x_t^g\}.
\end{equation*}
In case that this is a transitive action, the group $G$ is said to act
\emph{$t$-homogeneously \/}\index{group
action!t-homogeneous@$t$-homogeneous}\index{t-homogeneous group
action@$t$-homogeneous group action} on $X$.

Similarly, $G$ acts on the power set of $X$.

Later, we shall be concerned with $t$-homogeneous and $t$-transitive group
actions. Thus the following result, due to \textsc{Donald Livingstone} and
\textsc{Ascher Wagner} \cite{livi+w-65}, deserves our interest, even though we
are not going to use it.
\end{nrtxt}

\begin{thm}\label{thm:liv.wagner}
Suppose that the action of a group $G$ on a finite set $X$ is $t$-homogeneous,
where $4\leq 2t\leq\card X$. Then $G$ acts $(t-1)$-transitively on $X$. If,
moreover, $t>4$ then $G$ even acts $t$-transitively on $X$.
\end{thm}
\begin{txt}
See also \cite{wiel-67} for a short proof,  \cite[p.~92]{demb-97}, and
\cite[\S~16]{luen-69}.
\end{txt}

\begin{nrtxt}
An equivalence relation $\rel$ on $X$ is called \emph{$G$-invariant} if
\begin{equation}\label{eq:G-invar}
        x\rel y \Rightarrow x^g\rel y^g\mbox{ for all }x,y\in X
        \mbox{ and all }g\in G.
    \end{equation}
Then
\begin{equation}\label{eq:G-invar.scharf}
        x\rel y \Leftrightarrow x^g\rel y^g\mbox{ for all }x,y\in X
        \mbox{ and all }g\in G
\end{equation}
follows immediately, by applying (\ref{eq:G-invar}) to $x^g\rel y^g$ and
$g^{-1}$. The finest and the coarsest equivalence relation on $X$, i.~e.\ the
diagonal $\diag(X\times X)=\{(x,x)\mid x\in X\}$ and $X\times X$, obviously are
$G$-invariant equivalence relations on $X$.

Suppose now that $G$ acts transitively on $X$. If $\diag(X\times X)$ and
$X\times X$ are the only $G$-invariant equivalence relations on $X$ then the
action of $G$ is said to be \emph{primitive\/}\index{primitive group
action}\index{group action!primitive}; otherwise the action of $G$ is called
\emph{imprimitive\/}\index{group action!imprimitive}\index{imprimitive group
action}.

Suppose that $G$ acts imprimitively on $X$. A subset $S\subset X$ is called a
\emph{block of imprimitivity\/}\index{block!of
imprimitivity}\index{imprimitivity!block of} if it is an equivalence class of a
$G$-invariant equivalence relation, say $\rel$, which is neither $\diag(X\times
X)$ nor $X\times X$. Thus a block of imprimitivity is a subset $S$ of $X$ such
that $\card S>1$, $S\neq X$, and for all $g\in G$ we have either $S^g=S$ or
$S^g\cap S =\emptyset$.
\end{nrtxt}

\begin{nrtxt}
Given a subset $Y\subset X$ the \emph{setwise
stabiliser}\index{stabilizer!setwise}\index{setwise stabiliser} of $Y$ in $G$
is the set $G_Y$, say, of all $g \in G$ satisfying $Y^g=Y$. This stabiliser is
a subgroup of $G$. The \emph{pointwise stabiliser}\index{pointwise
stabiliser}\index{stabilizer!pointwise} of $Y\subset X$ in $G$ is the set of
all $g\in G$ such that $y^g=y$ for all $y\in Y$. This pointwise stabiliser is
also a subgroup of $G$ and, clearly, it is a normal subgroup of the setwise
stabiliser $G_Y$.

If $y\in X$ then we simply write $G_y$ instead of $G_{\{y\}}$. With this
convention, the mapping $y^g\mapsto G_y g$ is a bijection of the orbit $y^G$
onto the set of right cosets of $G_y$ in $G$, whence we obtain the fundamental
formula
\begin{equation}\label{eq:index.stabil}
        \card y^G =  \frac{\card G}{\card G_y}\,.
\end{equation}
It links cardinality of the orbit $y^G$ with the index of the stabiliser $G_y$
in $G$, i.~e.\ the number of right (or left) cosets of $G_y$ in $G$.

We refer to \cite[pp.~71--79]{jac-85} for a more systematic account on group
actions.
\end{nrtxt}

\section{A theorem of Spera}

\begin{nrtxt}
One possibility to construct divisible designs is given by the following
Theorem which is due to \textsc{Antonino Giorgio Spera}\index{theorem!of Spera}
\cite[Proposition~3.2]{spera-92a}. A similar construction for designs can be
found in \cite[Proposition~4.6]{beth+j+l-99a}.

The ingredients for this construction are a finite set $X$ with an equivalence
relation $\rel$ on its elements, a finite group $G$ acting on $X$, and a
so-called \emph{base block}\index{base block} (or \emph{starter
block}\index{starter block}\index{block!starter}) $B_0$, say. Its orbit under
the action of $G$ will then be our set of blocks. More precisely, we can show
the following:
\end{nrtxt}

\begin{thm}\label{thm:spera}
Let $X$ be a finite set which is endowed with an equivalence relation\/ $\rel$;
the corresponding partition is denoted by\/ $\cS$. Suppose, moreover, that $G$
is a group acting on $X$, and assume that the following properties hold:
\begin{enumerate}
    \item The equivalence relation\/ $\rel$ is $G$-invariant.

    \item All equivalence classes of\/ $\cR$ have the same cardinality, say $s$.
    \item
    The group $G$ acts transitively on the set of\/ $\rel$-transversal
    $t$-subsets of $X$ for some positive integer $t\leq\card\cS$.
\end{enumerate}
Finally, let $B_0$ be an $\rel$-transversal $k$-subset of $X$ with $t \le k$.
Then
\begin{equation*}
    (X,\cB,\cS) \mbox{ with } \cB:=B_0^G = \{B_0^g \mid g\in G\}
\end{equation*}
is a \tsk-divisible design, where
\begin{equation}\label{eq:DD-act-lambda}
  \lambda_t:=\frac{\card G}{\card G_{B_0}}\,
             \frac{ \Mat2{k\\t}}{\Mat2{vs^{-1}\\ t} s^t}\, ,
\end{equation}
and where $G_{B_0}\subset G$ denotes the setwise stabiliser of $B_0$.
\end{thm}

\begin{proof}
Firstly, let $\card X=:v$. Since $B_0$ is $\rel$-transversal, we have $0<t\leq
k=\card B_0\leq \card\cS=\frac vs$ so that axiom (D) in the definition of a DD
is satisfied. Also, we obtain $s,k>0$.

As $B_0$ is an $\rel$-transversal $k$-set, so is every element of $B_0^G$ by
(\ref{eq:G-invar.scharf}).  This verifies axiom (A), whereas axiom (B) is
trivially true due to assumption (b).

Next, to show axiom (C), we consider the base block $B_0$ and a $t$-subset
$Y\subset B_0$ which exists due to our assumption $t\leq k$. Let $\lambda_t>0$
be the number of blocks containing $Y$. Given an arbitrary $\cR$-transversal
$t$-subset $Y'\subset X$ there is a $g\in G$ with $Y'=Y^g$, since $Y\subset
B_0$ is $\cR$-transversal. This $g$ takes the $\lambda_t$ distinct blocks
through $Y$ to $\lambda_t$ distinct blocks through $Y'$. Similarly, the action
of $g^{-1}$ shows that there cannot be more than $\lambda_t$ blocks containing
$Y'$.

Altogether, we have verified the axioms of a divisible design. Yet, it remains
to calculate the parameter $\lambda_t$. By definition, the group $G$ acts
transitively on the set $\cB$ of blocks. By equation (\ref{eq:index.stabil}),
the number of blocks is
\begin{equation*}\label{eq:DD-act-b}
    b =  \frac{\card G}  {\card G_{B_0}}\,,
\end{equation*}
whence, by (\ref{eq:blockanzahl}), we get
\begin{equation*}
  \lambda_t = b\,\frac{\Mat2{k\\t}}{\Mat2{vs^{-1}\\ t} s^t}
  =\frac{\card G}{ \card G_{B_0}}\,\frac{\Mat2{k\\t}}{\Mat2{vs^{-1}\\ t} s^t}
\end{equation*}
which proves (\ref{eq:DD-act-lambda}).
\end{proof}
\begin{txt}
Note that in \cite{spera-92a} our condition (b) is missing. On the other hand
it is very easy to show that (b) cannot be dropped without effecting the
assertion of the theorem:
\end{txt}

\begin{exa}\label{bsp:gegen}
Let $X=\{1,2,3\}$, $\cS=\{\{1\},\{2,3\}\}$, and let $G$ be that subgroup of the
symmetric group $S_3$ which is formed by the identity $\id_X$ and the
transposition that interchanges $2$ with $3$. Then, apart from (b), all other
assumptions of Theorem \ref{thm:spera} are satisfied if we define $t:=2$ and
$B_0:=\{1,2\}$. However, no $2$-DD is obtained, since there are two blocks
containing $1$, but there exists only one block through the point $2$.
\end{exa}

\begin{nrtxt}
In the subsequent chapters we shall mainly apply a slightly modified version of
Theorem \ref{thm:spera} which is based on the following concept. A $t$-tuple
$(x_1,x_2,\ldots,x_t)\in X^t$ is called
\emph{$\cR$-transversal\/}\index{R-transversal t-tuple@$\rel$-transversal
$t$-tuple}\index{t-tuple@$t$-tuple!$\rel$-transversal} if its entries belong to
$t$ \emph{distinct\/} point classes.
\end{nrtxt}

\begin{cor}\label{cor:spera}
Theorem \emph{\ref{thm:spera}} remains true, mutatis mutandis, if assumption
\emph{(b)} is dropped and assumption \emph{(c)} is replaced with
\begin{itemize}
    \item[\emph{(c$_1$)}] The group $G$ acts transitively on the set of\/ $\rel$-transversal
    $t$-tuples of $X$ for some positive integer $t\leq\card\cS$.
\end{itemize}
\end{cor}

\begin{proof}
We observe that each $\rel$-transversal $t$-subset $Y$ gives rise to $t!$
mutually distinct $\rel$-trans\-versal $t$-tuples with entries from $Y$. As
$0<t\leq \card\cS$, it is obvious from (c$_1$) that $G$ acts transitively on
the set of $\rel$-transversal $t$-subsets of $X$, i.~e., condition (c) from
Theorem \ref{thm:spera} is satisfied.

In order to show that all equivalence classes of $\rel$ are of the same size,
we prove that $G$ acts transitively on $\cS$. Since assumption (a) remained
unchanged, formula (\ref{eq:G-invar.scharf}) can be shown as before. This
implies that, for all $S\in\cS$ and all $g\in G$, the image $S^g$ is an
equivalence class; hence $G$ acts on $\cS$. For this action to be transitive it
suffices to establish that $G$ operates transitively on $X$. So let $x_1$ and
$x'_1$ be arbitrary elements of $X$. We infer from $0<t\leq\card\cS$ that there
exist $\rel$-transversal $t$-tuples $(x_1,x_2,\ldots,x_t)$ and
$(x'_1,x'_2,\ldots,x'_t)$. By (c$_1$), there is at least one $g\in G$ which
takes the first to the second $t$-tuple. Therefore $x_1^g=x'_1$.
\end{proof}

\begin{nrtxt}
Suppose that a divisible design $\cD$ is defined according to Theorem
\ref{thm:spera} or Corollary \ref{cor:spera}. Then the action of $G$ on $\cS$
is $t$-homogeneous or $t$-transitive, respectively. In both cases the group $G$
acts on $X$ as an automorphism group of $\cD$ which, by the definition of
$\cB$, operates transitively on the set of blocks.
\end{nrtxt}

\begin{nrtxt}
Clearly, Theorem \ref{thm:spera} remains valid if we replace assumption (b)
with the following:
\begin{enumerate}
    \item[(b$_1$)] \emph{$G$ acts transitively on $\cS$}.
\end{enumerate}

Another possibility to alter the conditions in Theorem \ref{thm:spera} is as
follows \cite[Remark~2.1]{schulz+s-98a}: Suppose that condition (b) is dropped
and that (c) is replaced with
\begin{enumerate}\label{eq:spera.b2}
     \item[(c$_2$)] \emph{The group $G$ acts transitively on the set of\/ $\rel$-transversal
    $t$-subsets of $X$ for some positive integer $t<\card\cS$.}.
\end{enumerate}
In this case, let $y_1,y_2,\ldots, y_w$, where $w=\card\cS$, be a system of
representatives for the equivalence classes of $\cR$ such that
$\card[y_1]\leq\card[y_2]\leq\cdots\leq\card [y_w]$. We claim that (c$_2$)
implies
\begin{equation*}\label{}
   \card[y_1]=\card[y_2]=\cdots =\card [y_w]
\end{equation*}
which in turn is equivalent to (b). By (c$_2$), we have $0<t<w$ so that
\begin{equation*}
Y:=\{y_1,y_2,\ldots, y_t\} \mbox{~and~} Y':=\{y_2,y_3,\ldots, y_t,y_w\}
\end{equation*}
are $\rel$-transversal $t$-subsets of $X$. By the action of $G$ on $\cS$, the
$t$-tuple
\begin{equation*}
   \big({\card[y_2]},\card[y_3],\ldots,\card [y_t],\card [y_w]\big)
\end{equation*}
arises from
\begin{equation*}\label{}
  \big({\card[y_1]},\card[y_2],\ldots,\card [y_t]\big)
\end{equation*}
by re-arranging its entries. Therefore we obtain $\card[y_1]=\card [y_w]$, as
required.

Finally, we may even just drop assumption (b) if the integer $t$ admits the
application of Theorem \ref{thm:liv.wagner} which in turn will ensure that $G$
acts transitively on $\cS$.
\end{nrtxt}

\section{Divisible designs and constant weight codes}

\begin{nrtxt}
There is a close relationship between DDs and certain codes which will be
sketched in this section.

First, we collect some basic notions from coding theory. See, among others, the
book \cite{hill-86} for an introduction to this subject. Let us
write\footnote{In Chapter \ref{chap:P(R)} we shall use this symbol to denote
the ring of integers modulo $m$.}
\begin{equation*}\label{}
  \bZ_m:=\{0,1,\ldots,m\}\subset\bZ, \mbox{ where } m\geq 1.
\end{equation*}
Also let $n$ be a positive integer. The \emph{Hamming distance\/}\index{Hamming
distance} of $\vx=(x_1,x_2,\ldots,x_n)$ and
$\vy=(y_1,y_2,\ldots,y_n)\in\bZ_m^n$ is defined as the number of indices
$i\in\{1,2,\ldots,n\}$ such that $x_i\neq y_i$. It turns $\bZ_m^n$ into a
metric space. The \emph{Hamming weight\/}\index{Hamming weight} of an element
$\vx\in\bZ_m^n$ is its Hamming distance from $(0,0,\ldots,0)$ or, said
differently, the number of its non-zero entries. This terminology is in honour
of \textsc{Richard Wesley Hamming} (1915--1998), whose fundamental paper on
error-detecting and error-correcting codes appeared in 1950.

For our purposes it will be adequate to define an
\emph{automorphism\/}\index{automorphism!of $\bZ_m^n$} of $\bZ_m^n$ as a
product of any two mappings of the following form: First we apply a bijection
\begin{equation*}\label{}
   \bZ_m^n \to \bZ_m^n :
   (x_1,x_2,\ldots,x_n)\mapsto
   (x_1^{\alpha_1},x_2^{\alpha_2},\ldots,x_n^{\alpha_n}),
\end{equation*}
where each $\alpha_i$ is a permutation of $\bZ_m$, and then a bijection
\begin{equation*}\label{}
   \bZ_m^n \to \bZ_m^n :
   (x_1,x_2,\ldots,x_n)\mapsto
   (x_{1^\alpha},x_{2^\alpha},\ldots,x_{n^\alpha}),
\end{equation*}
where $\alpha$ is a permutation of $\{1,2,\ldots,n\}$. All such automorphisms
form a group under composition of mappings. Every automorphism preserves the
Hamming distance. The Hamming weight is preserved if, and only if,
$(0,0,\ldots,0)$ remains fixed.

An \emph{$m$-ary code\/}\index{code!$m$-ary}\index{m-ary code@$m$-ary code} of
\emph{length $n$}\index{length} is just a given subset $\vC\subset \bZ_m^n$.
Its elements are called \emph{codewords}\index{codeword}. The set $\bZ_m$ is
called the underlying \emph{alphabet}\index{alphabet} of the code $\vC$. A code
is called a \emph{constant weight code}\index{constant weight code} if all
codewords have the same (constant) Hamming weight.

Let $\vC_1,\vC_2\subset\bZ_m^n$ be codes. An
\emph{isomorphism\/}\index{isomorphism!of codes} is an automorphism of
$\bZ_m^n$ taking $\vC_1$ to $\vC_2$. An
\emph{automorphism\/}\index{automorphism!of a code}\index{code!automorphism of}
of a code is defined similarly.
\end{nrtxt}

\begin{nrtxt}
We now present the essential construction: Suppose that $\cD=(X,\cB,\cS)$ is a
\tsk-DD with $n:=\frac{v}{s}$ point classes. Also let $m:=s+1$. We augment $n$
\emph{ideal points}\index{ideal point}\index{point!ideal} to $X$, thus
obtaining a set $\widetilde X$ with
\begin{equation*}
    \card\widetilde X = v+n = mn.
\end{equation*}
To each point class we add precisely one ideal point in such a way that
distinct point classes are extended by distinct ideal points. Given a point
class $S\in\cS$ we write $\widetilde S$ for the corresponding \emph{extended
point class}\index{extended point class}\index{point class!extended}. Any block
$B\in\cB$ has $k\leq s$ points. We turn it into an \emph{extended
block}\index{extended block}\index{block!extended}, say $\widetilde B$, by
adding to $B$ the $n-k$ ideal points of those extended point classes
$\widetilde S$ which have empty intersection with $B$. Hence $\widetilde B$
meets every extended point class at precisely one point.

By the above, there exists a bijection
\begin{equation*}\label{}
   \psi: \widetilde X\to \{1,2,\ldots,n\}\times\bZ_m
\end{equation*}
such that for each point class $S\in\cS$ there is an index
$i\in\{1,2,\ldots,n\}$ with
\begin{equation*}\label{}
  S^\psi=\{i\}\times(\bZ_m\setminus\{0\})\mbox{ and }
  \widetilde S^\psi=\{i\}\times\bZ_m.
\end{equation*}
This means that under $\psi$ the set $\widetilde X\setminus X$ of ideal points
goes over to $\big\{(i,0)\mid i\in\{1,2,\ldots,n\}\big\}$. Furthermore, two
points of $\widetilde X$ are in the same extended point class if, and only if,
the first entries of their $\psi$-images coincide.

We are now in a position to define the \emph{code of\/}\index{code of a
DD}\index{DD!code of a} $\cD$ (with respect to $\psi$) as the subset of
$\bZ_m^n$ given by
\begin{equation*}\label{}
    \vC(\cD):=\{(j_1,j_2,\ldots,j_n)\mid \exists\, B\in\cB
                  : \widetilde B^\psi=\{(1,j_1),(2,j_2),\ldots,(n,j_n)\} \}.
\end{equation*}
According to our construction, all codewords have weight $k$, whence $\vC(\cD)$
is in fact a constant weight code.

In general, $\psi$ can be chosen in different ways. However, this will yield
isomorphic codes. So the actual choice of $\psi$ turns out to be immaterial. In
\cite{schulz+s-00} the codes arising in this way are characterised. Also, it is
shown that the entire construction can be reversed, i.~e., one can go back from
certain codes to divisible designs.
\end{nrtxt}

\begin{nrtxt}
A neat connection exists between the automorphism group of a DD and the
automorphism group of its constant weight code. Up to the exceptional case when
$t=2$ and $v=2k$, the two groups are isomorphic
\cite[Theorem~3.1]{schulz+s-00}. Also, if the automorphism group of $\cD$ is
``large'' then its corresponding code is well understood. See
\cite{etzion-97a}, \cite{schulz-98a}, and \cite{schulz+s-00} for a detailed
discussion.
\end{nrtxt}

\section{Notes and further references}

\begin{nrtxt}\label{:DD.lit1}
There is a widespread literature on divisible designs, and some particular
classes of DDs have been thoroughly investigated and characterised.

Among them are \emph{translation divisible designs}\index{translation DD},
i.~e.\ $2$-DDs with a group $T$ of automorphisms which acts sharply transitive
on $X$ (see \ref{:aktion.eigensch}) such that the following holds: For all
blocks $B\in\cB$ and all $g\in T$ there is either $B^g=B$ or $B^g\cap
B=\emptyset$. The name of these structures is due to the fact the same
properties hold, mutatis mutandis, for the action of the group of translations
on the set of points and lines of the Euclidean plane. We refer to
 \cite{bili+m-85a}, \cite{herz+s-89a},  \cite{jung-81a},
 \cite{schulz-84a}, \cite{schulz-85b}, \cite{schulz-85a},
 \cite{schulz-87a}, \cite{schulz-87b}, \cite{schulz-88a}, \cite{spera-90a},
 \cite{spera-91a},
and the references given there.

The more general class of ``$(s,k,\lambda_1,\lambda_2)$-translation DDs'' is
considered in \cite{schulz+s-92a} and \cite{spera-92b}.

Another construction of these more general DDs uses a \emph{Singer
group}\index{Singer group} with a \emph{relative difference
set}\index{difference set! relative}\index{relative difference set}
\cite{jung-82a}. As a general theme, each of the preceding constructions is
based upon a group which acts as a group of automorphisms of the DD.
\end{nrtxt}

\begin{nrtxt}\label{:DD.lit2}
While Theorem \ref{thm:spera} and Corollary \ref{cor:spera} pave the way to
constructing DDs, the actual choice of $X$, $\rel$, $G$, and a base block $B_0$
is a subtler question. We collect here some results:

In \cite{spera-92a} the following case is considered: $X$ is the projective
line over a finite local $K$-algebra $R$, and $G$ is the general linear group
$\GL_2(R)$ in two variables over $R$. All this is part of our exposition in
Chapters \ref{chap:P(R)} and \ref{chap:DD.GL}. In this way one obtains
$3$-divisible designs.

A higher-dimensional analogue, based upon the projective space over a finite
local algebra can be found in \cite{spera-95}; here, in general, only $2$-DDs
are obtained.

Another approach uses as the set $X$ the set of (affine) lines of a finite
translation plane, $\rel$ is chosen to be the usual parallelism of lines, and
$G$ is a group of affine collineations which acts $2$-transitively on the line
at infinity and contains all translations. Apart from the finite Desarguesian
planes this leads to L\"{u}neburg planes and Suzuki groups; see
\cite{schulz+s-98b} and \cite{spera-00a}. A more general setting, where $G$
acts $2$-transitively on a subset of the line at infinity can be found in the
papers \cite{cerr-02a}, \cite{cerr+sp-99}, and \cite{schulz+s-98a}.

A class of DDs, where $G$ is an orthogonal group or a unitary group, is
determined in \cite{cerr+sch-01a}. It was pointed out in \cite{giese+h+s-05a}
that one particular case of this construction is---up to isomorphism---a
Laguerre geometry (see \ref{:laguerre.geom}) which, by a completely different
approach, appears already in \cite{spera-92a}.

In \cite{cerr+sch-00a} the group $G$ is chosen to be the classical group
$\GL_3(q)$ (the general linear group in $3$ variables over the field with $q$
elements) in order to obtain divisible designs.

Also, we refer to \cite{spera-96a} for a discussion of transitive extensions of
imprimitive groups. A generalisation of Spera's construction was exhibited in
\cite{giese-05a} and \cite{giese+sch-07a}. It was put into a more general
context in \cite{blunck+h+z-07a} as follows: Let a group $G$ acting on some set
$X$ and a \emph{starter $t$-DD\/}\index{starter DD} in $X$ be given. Then,
under certain technical conditions, a new $t$-DD can be obtained via the action
of $G$ on $X$.

\end{nrtxt}


%% file: DD-Laguerre3.tex
\chapter{Laguerre Geometry}\label{chap:P(R)}

\section{Real Laguerre geometry}

\begin{nrtxt}
The classical \emph{Laguerre geometry\/}\index{Laguerre geometry} is the
geometry of spears and cycles in the Euclidean plane. A
\emph{spear\/}\index{spear} is an oriented line and a
\emph{cycle\/}\index{cycle} is either an oriented circle or a point (a ``circle
with radius zero''). There is a \emph{tangency relation\/}\index{tangency
relation} between spears and cycles; see the first two images in Figure
\ref{abb:zykel}. Furthermore, there exists a
\emph{parallelism\/}\index{parallelism}\index{parallel
spears}\index{spears!parallel} (written as $\parallel$) on the set of spears
which is depicted in the third image. We shall not give formal definitions of
these relations.

For our purposes it is more appropriate to identify a cycle with the set of all
its tangent spears. Then it is intuitively obvious that any cycle contains
precisely one spear from every parallel class, i.~e., it is a
``$\parallel$-transversal set''. Also, given any three non-parallel spears
there is a unique cycle containing them. All this reminds us of a divisible
design, even though the set of spears is infinite.

This geometry is named after the French mathematician \textsc{Edmond Nicolas
Laguerre} (1834--1886) who used it to solve a famous problem due to
\textsc{Apollonius of Perga} (262?--190?~BC): Find all circles that touch three
given circles (without orientation). See, for example, \cite{pedoe-72} and
\cite{pedoe-75}.
\begin{figure}[h]\begin{center}\unitlength1cm
\begin{picture}(11,3.76)
 \put(0,0){\includegraphics[width=11cm]{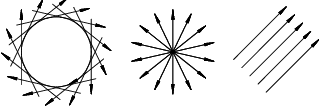}}
\end{picture}
\caption{\label{abb:zykel}Two cycles with tangent spears, and a family of
parallel spears}
\end{center}\end{figure}
\end{nrtxt}

\begin{nrtxt}
It was in the year of 1910 that \textsc{Wilhelm Blaschke} (1885--1962) showed
that the set of spears is in one-one correspondence with the points of a
circular cylinder of the Euclidean $3$-space \cite{blasch-10}, now called the
\emph{Blaschke cylinder\/}\index{Blaschke cylinder}\footnote{From the point of
view of projective geometry this is a quadratic cone without its vertex, whence
it is also called the \emph{Blaschke cone}\index{Blaschke cone}.}. Under this
mapping the cycles correspond to the ellipses on the cylinder and two spears
are parallel if, and only if, their images are on a common generator of the
cylinder. Blaschke also showed that the real Laguerre geometry can be
represented in terms of \emph{dual numbers\/}\index{dual
numbers}\index{numbers!dual} $x+y\eps$, where $x,y\in\bR$, $\eps\notin\bR$, and
$\eps^2=0$; see Example \ref{bsp:lokale.ringe}
(\ref{bsp:lokale.ringe.reelle.dualzahlen}) for a concise definition.
\end{nrtxt}

\begin{nrtxt}
There is a wealth of literature on the classical Laguerre geometry. We refer to
\cite[Chapter~1,\S~2]{benz-73}, \cite[Chapter~4]{benz-92},
\cite[Chapter~15~A]{giering-82}, \cite{rigby-81}, \cite{yaglom-79},
\cite{yaglom-81}, as well as the survey articles \cite{herz-95} and
\cite{schroeder-95}. Note that in \cite{yaglom-79} the term \emph{inversive
Galileian plane}\index{inversive Galileian plane}\index{Galileian
plane!inversive}---named after \textsc{Galileo Galilei} (1564--1642)---is used
instead.
\end{nrtxt}

\begin{nrtxt}
Our construction of divisible designs in Chapter \ref{chap:DD.GL} can be seen
as a generalisation of the classical Laguerre geometry, where a finite local
ring takes over the role of the ring of dual numbers over the reals. See
\cite{benz-07a} for a different generalisation of Laguerre geometry.
\end{nrtxt}

\section{The affine and the projective line over a ring}

\begin{txt}
\emph{All our rings are associative with a unit element \noslant{(}usually
denoted by $1$\noslant{)}, which is inherited by subrings and acts unitally on
modules. The trivial case $1=0$ is excluded.}
\end{txt}

\begin{nrtxt}
Let $R$ be a ring. Given an element $s\in R$ there are various possibilities:

If there is an $l\in R$ with $ls=1$ then $s$ is called \emph{left invertible}
\index{left invertible}\index{element!left invertible}\index{left invertible
element}. Such an element $l$ is said to be a \emph{left inverse}\index{left
inverse}\index{inverse!left} of $s$. \emph{Right invertible}
elements\index{right invertible}\index{element!right invertible}\index{right
invertible element} and \emph{right inverses} \index{right
inverse}\index{inverse!right} are defined analogously.

If $s$ has both a left inverse $l$ and a right inverse $r$ then
\begin{equation}\label{eq:li-re-inv}
  l=l1=l(sr)=(ls)r=1r=r.
\end{equation}
In this case, the element $s$ is said to be
\emph{invertible}\index{element!invertible}\index{invertible element}.
Moreover, by the above, all left (right) inverses of $s$ are equal to $r$ ($l$)
so that it is unambiguous to call $l=r=:s^{-1}$ the
\emph{inverse}\index{inverse} of $s$. The (multiplicative) group of invertible
elements (\emph{units}\index{unit}) of a ring $R$ will be denoted by $R^*$.
Clearly, $0$ is neither left nor right invertible.

If $s\neq 0$ then $s$ is called a \emph{left zero divisor}\index{zero
divisor!left}\index{left zero divisor} if there exists a non-zero element $r\in
R$ such that $sr=0$. Such an $s$ has no left inverse, since $ls=1$ would imply
$r=(ls)r=l(sr)=0$. However, an element without a left inverse is in general not
a left zero divisor. \emph{Right zero divisors}\index{zero
divisor!right}\index{right zero divisor} are defined similarly.

Of course the distinction between ``left'' and ``right'' is superfluous if $R$
is a commutative ring.
\end{nrtxt}

\begin{nrtxt}\label{:dedekind}
Suppose that we are given elements $a,b\in R$ with $ab=1$. Hence $y=y1=(ya)b$
for all $y\in R$. This implies that the \emph{right translation\/}\index{right
translation}\index{translation!right} $\rho_b:R\to R:x\mapsto xb$ is
surjective. Moreover, $(ba-1)b=b1-b=0$. Thus, whenever we are able to show that
$\rho_b$ is injective we obtain $ba-1=0$, i.~e., $ba=1$. This conclusion can be
applied, for example, if $R$ is a finite ring or a subring of the endomorphism
ring of a finite-dimensional vector space.

Rings with the property that, for all $a,b\in R$, $ab=1$ implies $ba=1$ are
called \emph{Dedekind-finite}\index{Dedekind-finite
ring}\index{ring!Dedekind-finite} (see e.g. \cite{lam-91}). In fact, in most of
our examples this condition will be satisfied. It carries the name of
\textsc{Richard Dedekind} (1831--1916).

\begin{exer}
Show that the endomorphism ring of an infinite dimensional vector space is not
Dedekind-finite.
\end{exer}

\end{nrtxt}

\begin{nrtxt}
Let $R$ be a ring. Then it is fairly obvious how to define the \emph{affine
line\/}\index{line!affine, over a ring}\index{affine line over a ring} over
$R$. It is simply the set $R$, but---as in real or complex analysis---we adopt
a geometric point of view by using the term \emph{point}\index{point} for the
elements of $R$. We shall meet again this affine line as a subset of the
projective line over $R$. However, to define something like a ``projective
line'' over a ring $R$ is a subtle task. As a matter of fact, various
definitions have been used in the literature during the last decades. Some of
those definitions are equivalent, some are equivalent only for certain classes
of rings. A short survey on this topic is included in \cite{lash-97}.
\end{nrtxt}

\begin{nrtxt}
Of course, a definition of a projective line over a ring has to include, as a
particular case, the projective line over a \emph{field}\index{field} $F$.
Observe that we use the term ``field'' for what other authors call a \emph{skew
field\/}\index{skew field}\index{field!skew} or a \emph{division
ring}\index{division ring}. Thus multiplication in a field need not be
commutative.

A particular case is well known from complex analysis: The \emph{complex
projective line}\index{projective line!complex}\index{complex projective
line}\index{line!complex projective} can be introduced as $\bC\cup\{\infty\}$,
where $\infty$ is an arbitrary new element. Intuitively, we think of $\infty$
as being $\frac a0$, where $a\in\bC$ is non-zero. For all $a\in \bC$, we have
$\frac a1 = \frac{xa}x$, $x\neq 0$. Thus \emph{every\/} fraction $\frac ab$
other than ``$\frac 00$'' determines an element of $\bC\cup\{\infty\}$.

It is immediate to carry this over to an arbitrary field $F$. However, one has
to be careful when using fractions in case that $F$ is non-commutative, since
$\frac ab$ could mean $ab^{-1}$ or $b^{-1}a$. We avoid ambiguity by
representing the elements of the \emph{projective line over an arbitrary
field\/}\index{line!projective, over a field}\index{projective line!over a
field} $F$ via
\begin{equation}\label{eq:P-F}
\begin{array}{rcl}
    a &\leftrightarrow& F(a,1)=\{x(a,1)\mid x\in F\} \mbox{ for all } a\in F,\\
    \infty&\leftrightarrow& F(1,0).
\end{array}
\end{equation}
More formally, the projective line over $F$ appears as the set of
one-dimensional subspaces of the \emph{left\/} vector space $F^2$. Every
non-zero vector in $F^2$ is a representative of a point. In terms of the
projective line the zero vector $(0,0)\in F^2$ has no meaning. Of course, we
could also consider $F^2$ as a \emph{right\/} vector space in order to describe
this projective line. The choice of ``left'' or ``right'' is just a matter of
taste.
\end{nrtxt}

\begin{nrtxt}
Now let us turn to an arbitrary ring $R$. We consider a (unitary) left module
$\vM$ over $R$. A family $(\vb_1,\vb_2,\ldots,\vb_n)$ of
\emph{vectors\/}\index{vector} in $\vM$ is called a \emph{basis\/}\index{basis}
provided that the mapping
\begin{equation}\label{eq:M.basis}
    R^n\to \vM : (x_1,x_2,\ldots,x_n)\mapsto\sum_{i=1}^n x_i\vb_i
\end{equation}
is a bijection. In this case $\vM$ is called \emph{free of rank}
$n$\index{module!free of rank $n$}\index{free of rank $n$}. It is important to
notice that this rank $n$ is in general \emph{not\/} uniquely determined by
$\vM$. See, for example, \cite[Example~1.4]{lam-99}

In order to define the projective line over a ring $R$ we start with a module
$\vM$ over $R$ which is free of rank $2$. By virtue of the bijection given in
(\ref{eq:M.basis}), we replace $\vM$ with $R^2$. Of course, the left $R$-module
$R^2$ is free of rank $2$; this is immediate by considering the \emph{standard
basis\/}\index{basis!standard}\index{standard basis} $\big((1,0),(0,1)\big)$ of
$R^2$.

It is tempting to define the projective line over a ring just in same way as we
did for a field in (\ref{eq:P-F}). However, this would not give ``enough
points'', since we would not get any ``point'' of the form $R(1,s)$, where
$s\neq 0$ has no left inverse. Nevertheless, $R(s,1)$ would be a point, i.~e.,
we would not have symmetry with respect to the order of coordinates. At the
other extreme one could say, as in the case of a field, that \emph{every\/}
pair $(a,b)\in R^2$, $(a,b)\neq (0,0)$ should be a representative of some
point. This point of view is adopted, for example, in
\cite[p.~1128]{bre+g+s-95}, where a distinction between ``points'' and ``free
points'' is made, and in \cite{faure-04a}. Yet, also here a problem arises: By
following this approach we would get, in general, ``far too much points'' for
our purposes.

It turned out that a ``good'' definition of the projective line over a ring $R$
is as follows: A submodule $R(a,b)\subset R^2$ is a point if $(a,b)$ is an
element of a basis with two elements. As in the case of a vector space, the
\emph{general linear group\/}\index{group!general linear}\index{general linear
group} $\GL_2(R)$ of invertible $2\times 2$-matrices with entries in $R$ acts
regularly on the set of those ordered bases of $R^2$ which consist of two
vectors. Therefore, starting at the canonical basis we are lead to the
following strict definition:
\end{nrtxt}

\begin{defi}\label{def:P(R)}
The \emph{projective line over $R$}\index{projective line!over a
ring}\index{line!projective, over a ring} is the orbit
\begin{equation*}
   \bP(R):=\big(R(1,0)\big)^{\GL_2(R)}
\end{equation*}
of $R(1,0)$ under the natural action of $\GL_2(R)$ on the subsets of $R^2$. Its
elements are called \emph{points}\index{point}.
\end{defi}
\begin{txt}
We refer to \cite[1.3]{blunck+he-05} and \cite[Definition~1.2.1]{herz-95} for
an equivalent definition which avoids using coordinates. Cf.\ also
\cite{blunck+h-01b} for the \emph{dual} of a \emph{projective
line}\index{projective line!dual of a}\index{dual of a projective
line}\index{line!dual of a projective}.
\end{txt}

\begin{nrtxt}
Let us describe $\bP(R)$ in different words: A pair $(a,b)\in R^2$ is called
\emph{admissible}\index{pair!admissible}\index{admissible pair} (over $R$) if
there exist $c,d\in R$ such that $\SMat2{a&b\\c&d}\in \GL_2(R)$. So we have
\begin{equation}
  \bP(R)=\{R(a,b)\subset R^2\mid (a,b)\mbox{ admissible}\}.
\end{equation}
Thus our definition of the projective line relies on admissible pairs. However,
there may also be non-admissible pairs $(a,b)\in R^2$ such that
$R(a,b)\in\bP(R)$. Strictly speaking, this phenomenon occurs precisely when $R$
is not Dedekind-finite (see \ref{:dedekind}). We refer to \cite{blunck+h-00b},
Propositions 2.1 and 2.2, for further details. We therefore adopt the following
convention:
\begin{center}
   \emph{Points of\/ $\bP(R)$ are represented by admissible pairs only}.
\end{center}
This brings us in a natural way to the next result:
\end{nrtxt}

\begin{thm}\label{thm:LRinvert}
Let $(a,b)\in R^2$ and $(a',b')$ be admissible pairs. Then $R(a,b)=R(a',b')$
if, and only if, there exists an element $u\in R^*$ with $(a',b')=u(a,b)$.
\end{thm}
\begin{proof}
Let $R(a,b)=R(a',b')$. By our assumption, there is a matrix $\gamma\in\GL_2(R)$
with first row $(a,b)$. Thus
\begin{equation*}
  (a,b)\cdot{\gamma^{-1}}=(1,0),\; (a',b')\cdot{\gamma^{-1}}=:(u,v),
  \mbox{ \ and \ }R(1,0)=R(u,v).
\end{equation*}
As $(a',b')$ is admissible, so is $(u,v)$. Now $(u,v)\in R(1,0)$ implies
$x(1,0)=(u,v)$ for some $x\in R$, whence $v=0$. Similarly, we obtain
$y(u,v)=(yu,0)=(1,0)$ for some $y\in R$. This means that $y$ is a left inverse
of $u$. By the above, $(u,v)=(u,0)$ is admissible. Hence there exists an
invertible matrix $\delta$, say, with first row $(u,0)$. Then
\begin{equation*}
        \Mat2{1 & 0\\0 &1 }=
        \underbrace{\Mat2{ u & 0 \\ {*} & {*}}}_{\delta}
        \cdot
        \underbrace{\Mat2{ z & {*} \\ {*} & {*}}}_{\delta^{-1}}
        =\Mat2{uz & {*}\\{*} &{*}}
\end{equation*}
shows that $z$, i.~e.\ the north-west entry of $\delta^{-1}$, is a right
inverse of $u$. Therefore
\begin{equation*}
 (a',b') = u\big((1,0)\cdot\gamma\big) = u(a,b)\mbox{ with } u\in R^*,
\end{equation*}
as required.

Conversely, if $u$ is a unit with $(a',b')=u(a,b)$ then $R=Ru$, whence
$R(a,b)=R(ua,ub)=R(a',b')$.
\end{proof}

\begin{nrtxt}\label{:viele.punkte}
We note that, for all $x\in R$,
\begin{equation}\label{eq:inv-matrizen}
         \Mat2{ x & 1 \\ {1} & {0}}
        =\Mat2{ 0 & 1 \\ {1} & {-x}}^{-1}\in\GL_2(R),\;
         \Mat2{ 1 & x \\ {0} & {1}}
        =\Mat2{ 1 & -x \\ {0} & {1}}^{-1}\in\GL_2(R).
\end{equation}
Hence the projective line over $R$ contains all points $R(x,1)$ with $x\in R$.
If $x,y\in R$ are different then $R(x,1)\neq R(y,1)$. Analogous results hold
for $R(1,x)\in\bP(R)$ for all $x\in R$. However, if $x\in R^*$ then
$R(1,x)=R(x^{-1},1)$, i.~e., this point is taken into account for a second
time. This shows that we can restrict ourselves to points $R(1,x)$ with $x\in
R\setminus R^*$, and it establishes the estimate
\begin{equation}\label{eq:mind.P(R)}
    \card\bP(R)\geq \card R + \card(R\setminus R^*).
\end{equation}
We shall see below that for certain rings the projective line contains even
more points. Cf.\ however Theorem \ref{thm:proj.gerade.lokal} and Corollary
\ref{cor:lokal.anzahl}.
\end{nrtxt}

\begin{exa}\label{bsp:Z6}
Let $\bZ/(6\bZ)=:\bZ_6$ be the (commutative) ring of integers modulo $6$. We
have $\bZ_6^*=\{1,5\}$, where $5\equiv-1\pmod 6$; the ideals of $\bZ_6$ are
$\{0\}$, $2\bZ_6=4\bZ_6$, $3\bZ_6$, and $\bZ_6$. Cf. \cite[2.6]{jac-85} for
further details.

As $x$ varies in $\bZ_6$, we obtain from the first matrix in
(\ref{eq:inv-matrizen}) six points
\begin{equation*}
  \bZ_6(0,1),\; \bZ_6(1,1),\; \ldots, \bZ_6(5,1),
\end{equation*}
and, for $x\in\bZ_6\setminus\bZ_6^*$ from the second part of
(\ref{eq:inv-matrizen}) four more points
\begin{equation*}
  \bZ_6(1,0),\; \bZ_6(1,2),\bZ_6(1,3),\;\bZ_6(1,4).
\end{equation*}
In this way we reach all points $\bZ_6(a,b)$ where $a$ or $b$ is a unit.
Therefore it remains to find out if there exist elements
$a,b\in\bZ_6\setminus\bZ_6^*$ and $c,d\in\bZ_6$ such that
\begin{equation*}
        \Mat2{ a & b \\ {c} & {d}}\in\GL_2(\bZ_6)
\end{equation*}
which in turn is equivalent to
\begin{equation*}
        \det\Mat2{ a & b \\ {c} & {d}}=ad-bc\in\bZ_6^*.
\end{equation*}
This means that the ideal generated by $a$ and $b$ has to be the entire ring
$\bZ_6$. Consequently,
\begin{equation*}
  (a,b)\in\{(2,3),(4,3),(3,2),(3,4)\}.
\end{equation*}
Thus the only remaining points in the projective line over $\bZ_6$ are
\begin{equation*}
  \bZ_6(2,3),\;\bZ_6(3,2).
\end{equation*}
Therefore $\card\bP(\bZ_6)=12$. Altogether, we see that among the $36$ elements
of $\bZ_6^2$ there are $24$ admissible and $12$ non-admissible pairs.
\end{exa}

\begin{nrtxt}\label{:unimod}
A pair $(a,b)\in R^2$ is called
\emph{unimodular\/}\index{pair!unimodular}\index{unimodular pair} (over $R$) if
there exist $x,y\in R$ with
\begin{equation*}
        ax+by=1.
\end{equation*}
 This is equivalent to saying that
the right ideal generated by $a$ and $b$ is the entire ring $R$.

Let $(a,b)$ be the first row of a matrix $\gamma\in\GL_2(R)$ and suppose that
the first column of $\gamma^{-1}$ reads $(x,y)^{\T}$. We read off from
$\gamma\gamma^{-1}=1$, where $1$ denotes the identity matrix in $\GL_2(R)$,
that every admissible pair is unimodular. We remark that
\begin{equation}\label{eq:uni-admiss}
  (a,b)\in R^2 \mbox{ unimodular over }R
  \Rightarrow (a,b)\mbox{ admissible over }R
\end{equation}
is satisfied, in particular, for all \emph{commutative\/} rings, since
$ax+by=1$ can be interpreted as the determinant of an invertible matrix with
first row $(a,b)$ and second row $(-y,x)$. \textsc{Walter Benz} in his famous
book \cite{benz-73} considers only commutative rings and defines the projective
line using unimodular pairs.

In fact (\ref{eq:uni-admiss}) holds also for certain non-commutative rings
\cite[Proposition~1.4.2]{herz-95}, namely for rings of \emph{stable
rank\/}\index{stable rank of a ring}\index{ring!stable rank of a} $2$, but we
shall not give a definition of this concept here. It was the late Dutch
geometer \textsc{Ferdinand~D.~Veldkamp} (1931--1999) who first pointed out the
significance for geometry of the stable rank of a ring. We refer to
\cite[\S~2]{veld-85} and \cite{veld-95} for excellent surveys on this topic.
Let us remark, however, that all finite rings are of stable rank $2$.

An example of a ring $R$, where (\ref{eq:uni-admiss}) is not true, can be found
in \cite[Remark~5.1]{blunck+h-01a}.
\end{nrtxt}

\begin{nrtxt}
As the concept of an admissible pair depends on the invertibility of square
matrices over a ring $R$, one may ask for a criterion which allows to decide
whether or not such a square matrix is invertible. In the general case,
something like this does not seem to exist. Nevertheless, there are particular
cases where we can not only decide invertibility but also explicitly describe
the inverse, as we already did in \ref{:viele.punkte}. Some of the subsequent
examples come from the \emph{elementary subgroup}\index{elementary
subgroup}\index{subgroup!elementary} of $\GL_2(R)$, i.~e.\ the subgroup
generated by elementary matrices; see \cite{cohn-66} for the algebraic
background, and \cite{blunck+h-01a} for the geometry behind.
\end{nrtxt}

\begin{exas}\label{bsp:matrizen}
Let $\gamma$ be a $2\times 2$ matrix over $R$.
\begin{enumerate}
\item If all entries of $\gamma$ commute with each other then we can calculate
the determinant $\det\gamma$ in the usual way. The given matrix is invertible
if, and only if, $\det\gamma\in R^*$. In this case $\gamma^{-1}$ can be
described in terms of $\det\gamma$ and the cofactor matrix of $\gamma$ as in
the case of a commutative field.

\item\label{bsp:matrizen.diag}

A diagonal matrix $\gamma=\diag(a,b)$ is invertible if, and only if $a$ and $b$
are units.

\item\label{bsp:matrizen.dreieck}

If we are given a lower triangular $2\times 2$ matrix $\gamma$ then
\begin{equation}\label{eq:dreieck}
    \gamma=: \Mat2{a&0\\c&d}
    =\Mat2{a&0\\0&1}
    \underbrace{\Mat2{1&0\\c&1}}_{\in\,\GL_2(R)}
    \Mat2{1&0\\0 &d}.
\end{equation}
We know from (\ref{eq:inv-matrizen}) that the second matrix on the right hand
side is invertible.

Suppose now that $a$ \emph{or} $d$ is a unit. By (\ref{bsp:matrizen.diag}) and
(\ref{eq:dreieck}), $\gamma$ is invertible if, and only if, $a$ \emph{and} $d$
are units. In this case
\begin{equation}\label{eq:dreicksinv}
    \gamma^{-1}=\Mat2{a^{-1}&0\\-d^{-1}ca^{-1}&d^{-1}}.
\end{equation}
Of course, there is a similar formula for the inverse of an upper triangular
matrix with invertible entries in the main diagonal.

\item\label{bsp:matrizen.ab=1}

Suppose that $a\in R$ is right invertible so that $ab=1$ for some $b\in R$. A
straightforward verification shows that
\begin{equation*}
\gamma:=\Mat2{a&0\\1-ba&b}\in\GL_2(R), \ \mbox{with} \
\gamma^{-1}=\Mat2{b&1-ba\\0&a}.
\end{equation*}
This means that for rings which are not Dedekind-finite there are invertible
\emph{lower\/} triangular matrices with \emph{both\/} diagonal entries not in
$R^*$. Also, somewhat surprisingly, the inverse of such a matrix is
\emph{upper\/} triangular.
\end{enumerate}
\end{exas}

\section{The distant relation}

\begin{nrtxt}\label{:distant}
The point set $\bP(R)$ is endowed with a relation
\emph{distant\/}\index{points!distant}\index{distant points} ($\dis$) which is
defined via the action of $\GL_2(R)$ on the set of pairs of points by
\begin{equation*}
  \dis:=\big(R(1,0), R(0,1)\big)^{\GL_2(R)}.
\end{equation*}
Letting $p=R(a,b)$ and $q= R(c,d)$ and taking into account
Theorem~\ref{thm:LRinvert} gives then
\begin{equation}\label{eq:dist-matrix}
  p\dis q\;\Leftrightarrow\; \Mat2{a&b\\c&d}\in \GL_2(R).
\end{equation}
The distant relation is symmetric, since exchanging two rows in an invertible
matrix does not influence its invertibility. In addition, $\dis$ is
anti-reflexive, because $R(1,0)\neq R(0,1)$ implies that distant points are
distinct\footnote{This is one of the rare occasions, where we need that $0\neq
1$ in $R$. Over the zero ring $R=\{0\}$ (which is excluded from our exposition)
we have $0=1$. Therefore, by defining the projective line as above, we obtain
$R(0,0)=R(1,0)=R(0,1)$. This means that $R(0,0)$ is the only point of this
projective line, and that $R(0,0)$ is distant to itself.}. However, in general
distinct points need not be distant. Cf.\ Theorem \ref{thm:koerper.dist} below.

Non-distant points ($\notdis$) are also called
\emph{neighbouring\/}\index{points!neighbouring}\index{neighbouring points} or
\emph{parallel}\index{parallel points}\index{points!parallel}; see, for
example, \cite{benz-73}, \cite{herz-95}, \cite{veld-95}. However, in these
lectures we shall use the term ``parallel'' in a different meaning which will
be explained in \ref{:parallel.allg}. The two notions ``parallel'' and
``neighbouring'' coincide precisely when $R$ is a local ring. See Theorem
\ref{thm:3} and our preliminary definition in \ref{def:parallel}.

A crucial property of the distant relation is stated in the following result on
the action of $\GL_2(R)$ on the projective line.
\end{nrtxt}

\begin{thm}\label{thm:3fach}
The group $\GL_2(R)$ acts \emph{$3$-$\dis$-transitively\/}\index{group
action!dis@$\dis$-transitive}\index{dis-transitive group
action@$\dis$-transitive group action} on $\bP(R)$, i.~e., transitively on the
set of triples of mutually distant points.
\end{thm}

\begin{proof}
(a) We consider the points $R(1,0)$, $R(0,1)$, and $R(1,1)$. They are
mutually distant by (\ref{eq:inv-matrizen}). Also, let $R(a,b)$ be a point
which is distant to $R(1,0)$ and $R(0,1)$. Consequently,
\begin{equation*}
\Mat2{1&0\\a&b}\in\GL_2(R)\mbox{~~and~~}\Mat2{a&b\\0 & 1}\in\GL_2(R).
\end{equation*}
Hence $a,b\in R^*$ by Example \ref{bsp:matrizen} (\ref{bsp:matrizen.dreieck}).
But this means that the matrix $\diag(a,b)\in\GL_2(R)$ takes $R(1,1)$ to
$R(a,b)$, whereas $R(1,0)$ and $R(0,1)$ remain unchanged.

(b)  Given three mutually distant points $p,q,r\in\bP(R)$ there is, by the
definition of the distant relation, a matrix $\gamma\in\GL_2(R)$ which takes
the pair of points $(p,q)$ to $\big(R(1,0),R(0,1)\big)$. Then, according to
(a), there is also an invertible matrix which takes $r^\gamma$ to $R(1,1)$,
while $R(1,0)$ and $R(0,1)$ remain invariant. Since this property holds for
every triple of mutually distant points, the assertion follows.
\end{proof}

\begin{nrtxt}
Let us determine the pointwise stabiliser $\Omega$, say, of
$\{R(1,0),R(0,1),R(1,1)\}$ under the action of $\GL_2(R)$ on the projective
line $\bP(R)$. If $\gamma$ is in this stabiliser then $\gamma=\diag(a,b)$,
because each of $R(1,0)$ and $R(0,1)$ has to coincide with its image. By
Example \ref{bsp:matrizen} (\ref{bsp:matrizen.diag}), $a$ and $b$ are units in
$R$. Moreover, we infer from $R(1,1)^\gamma=R(a,b)=R(1,1)$ that $a=b$. These
two conditions are also sufficient. Therefore
\begin{equation}\label{eq:stabilisator}
    \Omega=\{\diag(a,a)\mid a\in R^*\}.
\end{equation}
Now we ask for the kernel of the action of $\GL_2(R)$ on the projective line
$\bP(R)$ which clearly is contained in $\Omega$. If
$\gamma=\diag(a,a)\in\Omega$ is in this kernel then
\begin{equation*}
    R(1,x)^\gamma= R(a,xa)=R(1,a^{-1}xa)\mbox{~~for all~~}x\in R.
\end{equation*}
Recall that
\begin{equation*}
    \Z(R):=\{a\in R\mid ax=xa\mbox{~~for all~~}x\in R\}
\end{equation*}
is the \emph{centre}\index{centre!of a ring}\index{ring!centre of a} of $R$; it
is a subring of $R$. Therefore $a$ has to be unit in the centre of $R$.
Conversely, every matrix $\diag(a,a)$ with $a\in \Z(R)^*$ fixes $\bP(R)$
pointwise. It is easy to show (as in elementary linear algebra) that the kernel
of our group action is equal to the \emph{centre}\index{centre!of a
group}\index{group!centre of a} of $\GL_2(R)$, viz.\
\begin{eqnarray}\label{eq:zentrum.GL}\nonumber
    {\Z}\big({\GL_2(R)}\big)&=&\{\beta\in\GL_2(R)\mid \beta\xi =\xi\beta
                           \mbox{~~for all~~}\xi\in\GL_2(R)\}\\
                &\;=&\{\diag(a,a)\mid a\in \Z(R)^*\}.
\end{eqnarray}
As usual, the factor group $\GL_2(R) / {\Z}\big({\GL_2(R)}\big)=:\PGL_2(R)$ is
called a \emph{projective linear group}\index{group!projective
linear}\index{projective linear group}; it elements are called
\emph{projectivities\/}\index{projectivity} and can be considered as
permutations of $\bP(R)$.
\end{nrtxt}
\begin{thm}
The following statements are equivalent.
\begin{enumerate}
\item $\PGL_2(R)$ acts sharply transitive the set of triples of mutually
distant points.

\item The group $R^*$ of units in $R$ is contained in the centre $\Z(R)$.
\end{enumerate}
\end{thm}
\begin{proof}
The result is an immediate consequence of (\ref{eq:stabilisator}) and
(\ref{eq:zentrum.GL}).
\end{proof}

The interested reader should also compare this result with the
characterisations given in \cite[Proposition~1.3.4]{herz-95}.

\begin{nrtxt}\label{:nachbarschaft}
Given a point $p\in\bP(R)$ let
\begin{equation*}
  \dis(p):=\{x\in\bP(R)\mid x\dis p\}.
\end{equation*}
If we consider $\bP(R)$ as the set of vertices of the \emph{distant
graph}\index{distant graph}, i.~e.\ the unordered graph of the symmetric
relation $\dis$, then $\dis(p)$ is just the
\emph{neighbourhood}\index{neighbourhood} of $p$ in this graph. Once a point
$p$ has been chosen, the points of $\bP(R)$ fall into two classes: The points
of $\dis(p)$ are called \emph{proper\/}\index{proper point}\index{point!proper}
(with respect to $p$), the remaining points are called
\emph{improper\/}\index{improper point}\index{point!improper} (with respect to
$p$).

As $\GL_2(R)$ acts transitively on $\bP(R)$ it suffices to describe the
neighbourhood of $R(1,0)$, a point which is also denoted by the symbol
$\infty$. By Example \ref{bsp:matrizen} (\ref{bsp:matrizen.dreieck}), a point
$R(a,b)$ is in $\dis(\infty)$ precisely when $b\in R^*$. But then we may assume
w.l.o.g. that $b=1$, because $R(a,b)=R(b^{-1}a,1)$. The
\emph{embedding}\index{embedding}
\begin{equation}\label{eq:einbettung}
    R\to \bP(R) : a \mapsto R(a,1)
\end{equation}
maps the affine line over $R$ injectively onto the subset $\dis(\infty)$ of the
projective line over $R$. We already met this embedding in \ref{:viele.punkte}.
It shows that the neighbourhood of any point has $\card R$ elements.

By virtue of (\ref{eq:einbettung}), we may even identify the affine line over
the ring $R$ with the subset $\dis(\infty)$. From
\begin{equation*}
    \underbrace{\RMat2{1 & -1\\0 &1}}_{\in\,\GL_2(R)}
    \Mat2{a & 1\\b &1}=\Mat2{a-b & 0\\b &1}
\end{equation*}
follows that---in affine terms---two points $a,b\in R$ are distant, precisely
when $a-b$ is a unit.
\end{nrtxt}

\begin{exa}\label{bsp:Z6-dist}
We continue the investigation of the projective line $\bP(\bZ_6)$; see
Example \ref{bsp:Z6}.
\begin{figure}[h]\begin{center}\unitlength1cm\small
\begin{picture}(5,5.2) 
 \put(0,0){\includegraphics[width=6cm]{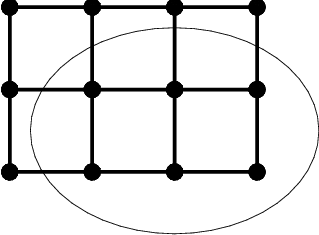}}
 \put(-0.25,4.55){$(1,0)$}
 \put(1.3,4.55){$(1,4)$}  
 \put(2.85,4.55){$(1,2)$}
 \put(4.4,4.55){$(3,2)$}

 \put(-0.9,2.6){$(1,3)$}
 \put(1.85,2.3){$(1,1)$}
 \put(3.4,2.3){$(5,1)$}
 \put(4.95,2.3){$(3,1)$}

 \put(-0.25,0.7){$(2,3)$}
 \put(1.3,0.7){$(4,1)$}
 \put(2.85,0.7){$(2,1)$}
 \put(4.4,0.7){$(0,1)$}

\end{picture}
\caption{\label{abb:Z6-dist}The distant relation on $\bP(\bZ_6)$}
\end{center}\end{figure}
In Figure \ref{abb:Z6-dist} each point of $\bP(\bZ_6)$ is labelled by one of
its admissible pairs. The distant relation on $\bP(\bZ_6)$ is illustrated in
the following way: Two distinct points are distant if they are \emph{not\/} on
a common line. The six points inside the ellipse comprise the neighbourhood of
$\infty=\bZ_6(1,0)$ in the distant graph.
\end{exa}

\begin{txt}
As a general theme, one aims at characterising algebraic properties of a ring
$R$ in terms of the distant relation on the associated projective line. Here is
a first result in this direction.
\end{txt}
\begin{thm}\label{thm:koerper.dist}
A ring $R$ is a field if, and only if, any two distinct points of the
projective line $\bP(R)$ are distant.
\end{thm}
\begin{proof}
(a) Let $R$ be a field. Given distinct points $p=R(a,b)$ and $q=R(c,d)$ of
$\bP(R)$ we obtain $(0,0)\neq(a,b)\notin R(c,d)$ and $(0,0)\neq(c,d)\notin
R(a,b)$, since a one-dimensional vector space is spanned by each of its
non-zero vectors. This means that $(a,b)$ and $(c,d)$ are linearly independent
vectors of the left vector space $R^2$, whence $p\dis q$ follows from
(\ref{eq:dist-matrix}).

(b) Conversely, the point $R(1,0)$ is distinct from each point $R(1,x)$, where
$x$ varies in $R\setminus\{0\}$. By Example \ref{bsp:matrizen}
(\ref{bsp:matrizen.dreieck}), we obtain that every non-zero element of $R$ is
invertible or, said differently, that $R$ is a field.
\end{proof}

\section{Chain geometries}

\begin{nrtxt}
The only structure on the projective line over a ring we have encountered so
far is the distant relation. Suppose now that a field $K$ is contained in $R$,
as a subring. Thus $1\in K$ is the identity element of $R$, and $R$ can be
considered as a left or a right vector space over $K$. The ring $R$ is, by
definition, a \emph{$K$-algebra\/}\index{K-algebra@$K$-algebra}\index{algebra}
precisely when the field $K$ belongs to the centre of $R$.
\end{nrtxt}

\begin{lem}\label{lem:einbett}
The mapping
\begin{equation}\label{eq:einbett}
  \bP(K)\to \bP(R): K(k,l)\mapsto R(k,l)
\end{equation}
is well defined. It takes distinct points of\/ $\bP(K)$ to distant points of\/
$\bP(R)$.
\end{lem}
\begin{proof}
The assertions are immediate from $\GL_2(K)\subset\GL_2(R)$.
\end{proof}
\begin{txt}
The following definition is taken from a paper by \textsc{Claudio Bartolone}
\cite{bart-89}. For a systematic account see \cite{blunck+h-00a}, and for the
particular case when $R$ is an algebra over $K$ the reader should compare with
\cite{benz-73}, \cite{blunck+he-05}, and \cite{herz-95}.
\end{txt}

\begin{defi}
Let $R$ be a ring containing a field $K$, as a subring. Also, let $C_0$ be the
image of the projective line $\bP(K)$ under the embedding (\ref{eq:einbett}). A
subset of $\bP(R)$ is called a
$K$-\emph{chain\/}\index{chain}\index{K-chain@$K$-chain} (or shortly a
\emph{chain}, $K$ being understood) if it belongs to set
\begin{equation*}
   \cC(K,R):=C_0^{\GL_2(R)}.
\end{equation*}
The \emph{chain geometry\/}\index{chain geometry} over $(K,R)$ is the structure
\begin{equation*}
  \Sigma(K,R) :=\big(\bP(R),\cC(K,R)\big).
\end{equation*}
\end{defi}
\begin{txt}
By definition, all chains arise from the \emph{standard chain\/}\index{standard
chain}\index{chain!standard} $C_0$ under the action of the group $\GL_2(R)$.
Observe that we refrain from excluding the trivial case when $R=K$.
\end{txt}

\begin{nrtxt}
If $\Sigma(K,R)$ and $\Sigma(K',R')$ are chain geometries then an
\emph{isomorphism\/}\index{isomorphism!of chain geometries}\index{chain
geometries!isomorphism of} is a bijection $\varphi:\bP(R)\to\bP(R')$ preserving
chains in both directions. By definition, the group $\PGL_2(R)$ is a group of
automorphisms of $\Sigma(K,R)$.

Our first observation is a characterisation of the distant relation $\dis$ of
$\bP(R)$ in terms of a chain geometry $\Sigma(K,R)$ (see \cite[2.4.2]{herz-95}
for the case of algebras):
\end{nrtxt}

\begin{thm}\label{thm:distant}
Let $p,q\in\bP(R)$ be distinct points of\/ $\Sigma(K,R)$. Then $p\dis q$ holds
if, and only if, there is a chain $D\in\cC(K,R)$ joining $p$ and $q$.
\end{thm}

\begin{proof}
By the definition of the distant relation in \ref{:distant}, we know that
$p\dis q$ implies $p=R(1,0)^\gamma$, $q=R(0,1)^\gamma$ for some $\gamma\in
\GL_2(R) $. Hence in this case $p,q\in C_0^\gamma\in\cC(K,R)$.

Conversely, if $p,q\in C_0^\gamma\in\cC(K,R)$, with $\gamma\in\GL_2(R) $, then
$p^{\gamma^{-1}}$ and $q^{\gamma^{-1}}$ are distinct points of the standard
chain $C_0=\bP(K)$. By Lemma \ref{lem:einbett}, we have $p^{\gamma^{-1}}\dis
q^{\gamma^{-1}}$. Since $\gamma$ preserves $\dis$, this proves the assertion.
\end{proof}

\begin{txt}
Given three mutually distant points we now want to determine the chains through
them. Note that, by Theorem \ref{thm:distant}, any two distinct points on a
chain are distant.
\end{txt}

\begin{thm}\label{thm:3punkte}
Let the points $p,q,r\in \bP(R)$ be mutually distant. Then there is at least
one chain $D\in\cC(K,R)$ containing $p$, $q$, and $r$.
\end{thm}

\begin{proof}
As the group $\GL_2(R) $ acts $3$-$\dis$-transitively on $\bP(R)$ by Theorem
\ref{thm:3fach}, there exists a $\gamma\in\GL_2(R)$ with $p=R(1,0)^\gamma$,
$q=R(0,1)^\gamma$, and $r=R(1,1)^\gamma$. Obviously, $D:=C_0^\gamma$ is a chain
through $p$, $q$, and $r$.
\end{proof}

\begin{txt}
The essential result on the group action of $\GL_2(R) $ on $\Sigma(K,R)$ is as
follows:
\end{txt}

\begin{thm}\label{thm:3fach.kette}
Let $D,D'\in\cC(K,R)$ be chains. Suppose, furthermore, that $p,q,r\in D$ and
$p',q',r'\in D'$ are, respectively, three mutually distinct points. Then there
exists a matrix $\gamma\in\GL_2(R) $ such that $p^\gamma=p'$, $q^\gamma=q'$,
$r^\gamma=r'$, and $ D^\gamma= D'$.
\end{thm}

\begin{proof}
There exists a matrix $\gamma_1\in\GL_2(R) $ mapping $D$ to the standard chain
$C_0$. Put $p_1:=p^{\gamma_1}$, $q_1:=q^{\gamma_1}$, $r_1:=r^{\gamma_1}$. The
group $\GL_2(K)\subset\GL_2(R)$ leaves $C_0$ invariant and acts $3$-fold
transitively on $C_0$. Hence there is a $\gamma_2\in\GL_2(K)$ with
$p_1^{\gamma_2}=R(1,0)$, $q_1^{\gamma_2}=R(0,1)$, $r_1^{\gamma_2}=R(1,1)$.
Then, we also have $C_0^{\gamma_2}=C_0$.

Define $\gamma_1'$ and $\gamma_2'$ accordingly. Then
$\gamma=\gamma_1\gamma_2\gamma_2'^{-1}\gamma_1'^{-1}$ has the required
properties.
\end{proof}

Now it is easy to determine the number of chains containing three mutually
distant points:

\begin{thm}\label{thm:normalisator}
Let
\begin{equation*}
   N:=\{ n\in R^* \mid n^{-1}K^*n=K^* \}
\end{equation*}
be the \emph{normaliser}\index{normaliser} of $K^*$ in $R^*$. Then the
following assertions hold:
\begin{enumerate}

\item\label{thm:normalisator.a}

The set of chains through any three mutually distant points of\/ $\Sigma(K,R)$
is in $1$-$1$-correspondence with the set
\begin{equation*}
   \{Nr\mid r\in R^*\}
\end{equation*}
of right cosets of $N$ in the multiplicative group $R^*$.

\item\label{thm:normalisator.b}

In\/ $\Sigma(K,R)$ there exists exactly one chain through any three mutually
distant points if, and only if, $K^*$ is a normal subgroup of $R^*$.
\end{enumerate}
\end{thm}

\begin{proof} We recall from (\ref{eq:stabilisator}) that the subgroup
\begin{equation*}
   \Omega=\{\diag(a,a)\mid a\in R^*\}\cong R^*
\end{equation*}
of $\GL_2(R) $ is the pointwise stabiliser of the set $\{R(1,0), R(0,1),
R(1,1)\}$. So, by Theorem \ref{thm:3fach.kette}, the chains through $R(1,0)$,
$R(0,1)$, $R(1,1)$ are precisely the images $C_0^\omega$, where $\omega$ ranges
in $\Omega$. Since
\begin{equation*}
  R(1,x)^{\Omega} = \{ R(1,a^{-1}xa)\mid a\in R^*\}
\end{equation*}
holds, in particular, for all $x\in K^*$, the stabiliser of the standard chain
$C_0$ in $\Omega$ is
\begin{equation*}
   \Omega_{C_0}=\{\diag(n,n)\mid n\in N\}\cong N.
\end{equation*}
So, by (\ref{eq:index.stabil}), assertion (a) follows for the three given
points and, by Theorem \ref{thm:3fach.kette}, for any three pairwise distant
points.

Of course, the condition in (\ref{thm:normalisator.b}) just means that $R^*=N$.
 \end{proof}

\begin{exas}\label{bsp:1kette}
In each of the following examples there is a unique chain through any three
distinct points of $\Sigma(K,R)$:
\begin{enumerate}
\item\label{bsp:1kette.algebra} Suppose that $K$ belongs to the centre of
    $R$, i.~e., $R$ is a $K$-algebra. Then, since $K^*$ is in the centre of
    $R^*$, its normaliser $N$ coincides with $R^*$. Most of the examples
    which we shall encounter later on will be of this kind.

\item Let $R$ be a commutative ring. Then the assumptions of Example
(\ref{bsp:1kette.algebra}) are satisfied without imposing a condition on $K$.

\item Suppose that $K^*=R^*$. Then $N=R^*=K^*$ is trivially true. Observe
    that $K^*=R^*$ does not mean that $K=R$; take, for example, a
    polynomial ring $K[T]$ over a commutative field $K$ in an indeterminate
    $T$; see also \cite[Example~2.5~(a)]{blunck+h-00a}.

\item\label{bsp:1kette.GF.4} Let $\bZ_2=\GF(2)$ be the field with two elements.
Also let $R=\bZ_2^{2\times 2}$ be the ring of $2\times 2$ matrices over
$\bZ_2$. There are six invertible elements in this ring, namely
\begin{equation*}
  \Mat2{1&0\\0&1},\;\Mat2{1&0\\1&1},\;\Mat2{0&1\\1&0},\;\Mat2{0&1\\1&1},\;
  \Mat2{1&1\\1&0},\;\Mat2{1&1\\0&1}.
\end{equation*}
The centre of $R$ is given by $\Z(R)=\{\diag(x,x)\mid x\in\bZ_2\}$. We put
\begin{equation*}
    K:=\left\{\left.\Mat2{x & y\\ y & x+y} \;\right|\; x,y\in\bZ_2 \right\}.
\end{equation*}
It is easily seen that $K$ is a subring of $R$ which is isomorphic to
$\GF(4)$, i.~e.\ the field with $4$ elements. Of course,
$K^*\not\subset\Z(R)$. Since $\card R^*=6$, the multiplicative group $K^*$
has index $2$ in $R^*$ and therefore is normal.
\end{enumerate}
\end{exas}

\begin{txt}
We now determine the intersection of all chains through three mutually distant
points of a chain geometry $\Sigma(K,R)$. To this end we introduce the field
\begin{equation*}
     F:=\bigcap\limits_{a\in R^*}a^{-1}Ka
\end{equation*}
which is a subring of $R$. Consequently, we can embed the projective line
$\bP(F)$ in $\bP(R)$ and define a chain geometry $\Sigma(F,R)$ as above. Its
chains will be called \emph{$F$-chains\/} in order to distinguish them from the
chains which arise from $\Sigma(K,R)$.
\end{txt}

\begin{thm}
Let $p,q,r\in\bP(R)$ be mutually distant points. Then the intersection of all
chains of\/ $\Sigma(K,R)$ through $p,q,r$ is an $F$-chain.
\end{thm}
\begin{proof} We consider w.l.o.g.\ the points $R(1,0)$, $R(0,1)$, and $R(1,1)$.
According to Theorem \ref{thm:3fach.kette} the chains joining them are exactly
the images $C_0^\omega$, with $\omega\in\Omega$; compare
(\ref{eq:stabilisator}). Given a matrix $\diag(a,a)\in\Omega$ we compute
\begin{equation*}
  C_0^\omega
=\{R(a,0)\}\cup\{R(ka,a)\mid k\in K\} =\{R(1,0)\}\cup\{R(a^{-1}ka,1)\mid k\in
K\}.
\end{equation*}
Therefore
\begin{eqnarray*}
  \bigcap\limits_{\omega\in\Omega} C_0^\omega
  &=&
  \{R(1,0)\}\cup\bigcap\limits_{a\in R^*}\{R(a^{-1}ka,1)\mid k\in K\}\\
  &=&
  \{R(1,0)\}\cup\{R(f,1)\mid f\in F\},
\end{eqnarray*}
which equals $\bP(F)$, considered as a subset of $\bP(R)$.
\end{proof}

\section{Local rings, local algebras, and Laguerre algebras}

\begin{nrtxt}\label{:jacobson}
Let $R$ be a ring. The \emph{Jacobson radical}\index{Jacobson radical} of a
ring $R$, named after \textsc{Nathan Jacobson} (1910--1999) and denoted by
$\rad R$, is the intersection of all maximal left (or right) ideals of $R$. It
is a two sided ideal of $R$ and its elements can be characterised as follows:
\begin{equation*}
 b\in\rad\,R \,\Leftrightarrow\, 1-ab\in R^*\mbox{ for all }a\in R
 \,\Leftrightarrow\, 1-ba\in R^*\mbox{ for all }a\in R;
\end{equation*}
see \cite[pp.~53--54]{lam-91}.

Suppose that $R$ is \emph{left artinian}\index{left artinian
ring}\index{ring!left artinian}---after \textsc{Emil Artin}
(1898--1962)---i.~e., there does not exist an infinite strictly descending
chain of left ideals of $R$. then $\rad R$ is the largest \emph{nilpotent} left
ideal\index{nilpotent left (right) ideal}\index{ideal!nilpotent left (right)},
and it is also the largest nilpotent right ideal; this means that $(\rad
R)^n=0$ for some positive integer $n$ \cite[Theorem~4.12]{lam-91}.
Consequently, $\rad R$ is actually a nilpotent ideal. All this holds, in
particular, if $R$ is a finite ring. See \cite{kruse+p-69} for further
references on nilpotent rings.
\end{nrtxt}

\begin{nrtxt}\label{:local.ring}
A ring $R$ is called a \emph{local ring}\index{ring!local}\index{local ring} if
$R\setminus R^*$ is an ideal\footnote{By an ``ideal'' we always mean a
two-sided ideal. The term ``local ring'' comes from algebraic geometry: At any
point $p$ of an algebraic variety, the rational functions which are ``locally''
regular (i.~e.\ regular in some neighbourhood of $p$) form a local ring. The
non-units in this ring are those functions which vanish at $p$. Compare
\cite[p.~72]{shaf-77}.} of $R$. There are several equivalent definitions of a
local ring and the interested reader should compare with
\cite[Theorem~19.1]{lam-91}. We just mention that a ring $R$ is local if, and
only if, it has an ideal $J\neq R$ containing all ideals other than $R$. This
is equivalent to saying that $R$ has a unique maximal ideal.

Let $R$ be a local ring. Since $R\setminus R^*$ is the only maximal left ideal
of $R$, we obtain
\begin{equation*}
  \rad R=R\setminus R^*,
\end{equation*}
Since $\rad R$ is an ideal, we can construct the factor ring $\overline
R:=R/\rad R$ based upon the canonical epimorphism $R\to\overline R: a\mapsto
\overline a$. If $\overline a\neq \overline 0$ then $a\in R^*$, whence
$\overline a$ is a unit in $\overline R$. This means that $\overline R$ is a
field, and we have the property
\begin{equation}\label{eq:a-quer}
    a\in R^* \;\Leftrightarrow\; \overline a \in \overline R^*.
\end{equation}
Given a matrix $\gamma=(\gamma_{ij})$ with entries in $R$ we put
$\overline\gamma:=(\overline{\gamma_{ij}})$. Then one can show as above that
\begin{equation}\label{eq:gamma-quer}
    \gamma\in \GL_m(R) \;\Leftrightarrow\;
    \overline \gamma \in \GL_m(\overline R)
\end{equation}
holds for all natural numbers $m\geq 1$.
\end{nrtxt}

\begin{nrtxt}
A $K$-algebra $R$ is said to be \emph{local\/}\index{local
algebra}\index{algebra!local} if $R$ is a local ring. Clearly, $K$ and $\rad R$
are subspaces of the vector space $R$ (over $K$), and they meet at $0$ only.
If, moreover, the group $(R,+)$ is the direct sum of its subgroups $K$ and
$\rad R$ then $R$ is called a \emph{Laguerre algebra} over $K$. Here it is
important to emphasise the ground field. Each Laguerre algebra $R$ over $K$ is
a local algebra over any \emph{proper\/} subfield $F$ of $K$. On the other
hand, it is not a Laguerre algebra over $F$, because $F\oplus\rad R$ (direct
sum of additive groups) is a proper subgroup of $(R,+)$.
\end{nrtxt}

\begin{exas}\label{bsp:lokale.ringe}
Here are some examples of local rings and local algebras:
\begin{enumerate}

\item\label{bsp:lokale.ringe.koerper}

A trivial example of a local ring is a field.

\item\label{bsp:lokale.ringe.reelle.dualzahlen}

As has been noted before, the classical example of a local ring is the ring of
\emph{dual numbers}\index{dual numbers}\index{numbers!dual} over the reals.
There are several ways to define it. For example, we may start with the
polynomial ring $\bR[T]$ in the indeterminate $T$, consider the ideal $( T^2)$
which is generated by $T^2$, and define the real dual numbers as the quotient
ring $\bR[T]/(T^2)$. Letting $\eps:=T+(T^2)$ leads to the usual notation of a
dual number in the form
\begin{equation*}
    a+ b \eps \mbox{~~with~~} a,b\in\bR,
    \mbox{~~where~~}\eps\notin \bR, \mbox{~~and~~}\eps^2= 0.
\end{equation*}
This example allows several generalisations which are discussed below.

\item\label{bsp:lokale.ringe.dualzahlen}

In Example (\ref{bsp:lokale.ringe.reelle.dualzahlen}) we may replace $\bR$ with
any commutative field $K$ thus obtaining the ring of \emph{dual numbers\/} over
$K$. Such a ring of dual numbers will be denoted by $K[\eps]$. It is a
two-dimensional Laguerre algebra over $K$ with $\rad K[\eps]=K\eps$.

We may even allow $K$ to be a (non-commutative) field if we require $T$ to be a
\emph{central indeterminate}\index{indeterminate!central}\index{central
indeterminate}. This means that in the polynomial ring $K[T]$ the indeterminate
$T$ commutes with every element of $K$. Even though this ring of dual numbers
is of the form $K\oplus K\eps$, it is not an algebra over $K$, unless $K$ is
commutative.

\item\label{bsp:lokale.ringe.twist}

Let $R=K[\eps]$ be a ring of dual numbers as in
(\ref{bsp:lokale.ringe.dualzahlen}) and let $\sigma\in\Aut(K)$ be an
automorphism of $K$ other than the identity. We keep addition unaltered, but
introduce a new multiplication (denoted by $*$) in $K[\eps]$ as follows:
\begin{equation*}
    (a+b\eps)*(c+d\eps) := ac + (ad+bc^\sigma)\eps
    \mbox{~~for all~~}a,b,c,d\in K.
\end{equation*}
This gives a ring $K[\eps;\sigma]$ of \emph{twisted dual
numbers}\index{numbers!twisted dual}\index{twisted dual numbers} over $K$. It
is a local ring with $K\eps$ the ideal of all non-invertible elements. It
cannot be an algebra over $K$, even if $K$ is commutative, because $K$ is not
in the centre of $K[\eps;\sigma]$.

\item\label{bsp:lokale.ringe.reelle.hoeheredualzahlen}

An immediate generalisation of (\ref{bsp:lokale.ringe.dualzahlen}) is to
consider the factor ring $K[T]/(T^h)$ for some natural number $h\geq 1$. As
before, we put $\eps:=T+(T^h)$, whence this ring is of the form
\begin{equation*}
  K[\eps]:= K\oplus \underbrace{K\eps\oplus K\eps^{2}\oplus \ldots\oplus K\eps^{h-1}}_{=\,\rad K[\eps]}.
\end{equation*}

\item\label{bsp:lokale.ringe.aeussere}

Let $\vV$ be an $n$-dimensional vector space over a commutative field $K$. Then
the \emph{exterior algebra}\index{algebra!exterior}\index{exterior algebra}
\begin{equation}\label{eq:aeussere}
    \bigwedge \vV = \underbrace{\bigwedge\nolimits^0 \vV}_{=\,K} \oplus
    \underbrace{\bigwedge\nolimits^1 \vV}_{=\, \vV}
    \oplus \cdots\oplus
    \bigwedge\nolimits^{n} \vV
\end{equation}
is a Laguerre algebra over $K$ with dimension $2^n$; see, for example,
\cite[7.2]{jac-85}. Multiplication in this algebra is usually denoted by the
wedge sign ($\wedge$). If $(\vb_1,\vb_2,\ldots,\vb_n)$ is a basis of $\vV$ then
the family of vectors
\begin{equation*}
  \vb_{i_1}\wedge\vb_{i_2}\wedge\ldots\wedge\vb_{i_k},
\end{equation*}
where $1\leq i_1<i_2<\cdots<i_k\leq n$ and $k\in\{0,1,\ldots,n\}$ is a
basis of $\bigwedge\vV$. Of course, when $k=0$ the corresponding empty
product is defined to be $1\in \bigwedge\vV$. The product of vectors is
alternating and therefore skew symmetric. Thus we have $\vv\wedge \vv=0$
and $\vv\wedge\vw=-\vw\wedge\vv$ for all $\vv,\vw\in\vV$.

In particular, for $\vV=K$ the exterior algebra $\bigwedge K$ is just the ring
of dual numbers over $K$. Here some care has to be taken, since according to
(\ref{eq:aeussere}) we get two copies of $K$ in $\bigwedge K$, namely
$\bigwedge\nolimits^0 K$ (a copy of the field $K$) and $\bigwedge\nolimits^1 K$
(a copy of the vector space $K$), and they \emph{must not be identified}.

\item\label{bsp:lokale.ringe.Zmod}

Let $\bZ$ be the ring of integers and let $1<q=p^h\in\bZ$ be a power of a
prime $p$. Then $\bZ/(q\bZ) =:\bZ_q$ is a local ring. The ideal $\rad
\bZ_q$ comprises the residue classes (modulo $q$) of all integers $kp$,
where $k\in\bZ$, so it is the zero ideal precisely when $h=1$. The quotient
field $\bZ_q/\rad R$ is the \emph{Galois field}\index{Galois field}
$\bZ_p=\GF(p)$ which carries the name of \textsc{Evariste Galois}
(1811--1832).

If $h>1$ then $\bZ_q$ is not an algebra over any field, because the smallest
positive integer $n$ satisfying $\sum_{i=1}^n 1\equiv 0\pmod q$ is $n=q$.
However, the characteristic of a finite field is a prime, and an infinite field
cannot be a subset of $\bZ_q$.
\end{enumerate}
\end{exas}

\begin{txt}
While for an arbitrary ring it is difficult (or maybe even hopeless) to
describe explicitly the associated projective line, for a local ring this is an
easy task:
\end{txt}

\begin{thm}\label{thm:proj.gerade.lokal}
Let $R$ be a local ring. Then
\begin{equation}\label{eq:proj.gerade.lokal}
    \bP(R)= \{R(x,1)\mid x\in R\}\cup \{R(1,x)\mid x\in R\setminus R^*\}.
\end{equation}
\end{thm}
\begin{proof}
By \ref{:viele.punkte}, the elements of the sets on the right hand side of
(\ref{eq:proj.gerade.lokal}) are points of $\bP(R)$. We infer from
(\ref{eq:gamma-quer}) that the mapping
\begin{equation}\label{eq:P(R)->P(Rquer)}
    \bP(R)\to\bP(\overline R): R(a,b)\mapsto R(\overline a,\overline b)
\end{equation}
is well-defined; moreover, it takes distant points of $\bP(R)$ to distinct
points of the projective line over the field $\overline R$. Cf.\ Theorem
\ref{thm:koerper.dist}. So let $R(a,b)$ be a point of $\bP(R)$. By
(\ref{eq:P(R)->P(Rquer)}), $\overline R(\overline a,\overline b)$ is a point of
$\bP(\overline R)$. Thus either $\overline b\neq\overline 0$, whence $b\in R^*$
and $R(a,b)=R(b^{-1}a,1)$; or $\overline b=\overline 0$, whence $\overline
a\neq \overline 0$, $b\in R\setminus R^*$, $a\in R^*$, and
$R(a,b)=R(1,a^{-1}b)$.
\end{proof}

\begin{cor}\label{cor:lokal.anzahl}
The projective line over a local ring $R$ has cardinality
\begin{equation}\label{eq:lokal.anzahl}
    \card\bP(R)= \card R + \card \rad R.
\end{equation}
\end{cor}
\begin{txt}
This improves formula (\ref{eq:mind.P(R)}) for local rings. The following
characterisation is essential:
\end{txt}
\begin{thm}\label{thm:lokal.dist}
A ring $R$ is a local ring if, and only if, the relation ``non-distant''
\noslant($\notdis$\noslant) on the projective line $\bP(R)$ is an equivalence
relation.
\end{thm}
\begin{proof}
(a) Over any ring $R$, the relation $\notdis$ on $\bP(R)$ is reflexive and
symmetric, since $\dis$ is anti-reflexive and symmetric according to
\ref{:distant}.

(b) Suppose that $R$ is local. By the action of $\GL_2(R)$, it suffices to show
that $p\notdis R(1,0)$ and $R(1,0)\notdis q$ implies $p\notdis q$ for all
$p,q\in\bP(R)$. With $p=R(a,b)$ and $q=R(c,d)$ we obtain
\begin{equation*}
    \Mat2{1 & 0 \\ a & b}\notin\GL_2(R)\mbox{~~and~~}
    \Mat2{1 & 0 \\ c & d}\notin\GL_2(R).
\end{equation*}
Thus, by Example \ref{bsp:matrizen} (\ref{bsp:matrizen.dreieck}), $b$ and $d$
are in $R\setminus R^*=\rad R$. But then $xb+yd\in\rad R$ for all $x,y\in R$,
whence
\begin{equation*}
    \Mat2{* & * \\ x & y}\Mat2{a & b \\ c & d}\neq \Mat2{* & * \\ * & 1}
    \mbox{~for all~} x,y\in R.
\end{equation*}
This implies $p\notdis q$.

(c) Conversely, let $\notdis$ be an equivalence relation. We have to show that
$J:=R\setminus R^*\neq\emptyset$ is an ideal. Given $a,b\in J$ we infer from
\ref{:nachbarschaft} that $ R(1,a)\notdis R(1,0)\notdis R(1,b)$. So,
transitivity of $\notdis$ yields
\begin{equation*}
   \Mat2{1& a\\1& b} \notin\GL_2(R).
\end{equation*}
From
\begin{equation*}
   \Mat2{1& a\\1& b} =
   \Mat2{1& 0\\1& a-b}\underbrace{\Mat2{1& a\\0& -1}}_{\in\,\GL_2(R)}
   \notin\GL_2(R)
\end{equation*}
we read off that the first matrix on the right hand side is not invertible,
whence $a-b\in J$. Thus $J$ is an additive subgroup of $R$.

Next, we show that $ab=u$, where $a,b\in R$ and $u\in R^*$, implies that $a$
and $b$ are units. It suffices to treat the case $u=1$: By Example
\ref{bsp:matrizen} (\ref{bsp:matrizen.ab=1}), the matrix
\begin{equation*}
    \Mat2{a&0\\1-ba&b}
\end{equation*}
has an inverse. Hence $R(a,0)$ and $R(1-ba,b)$ are points such that
\begin{equation*}
  R(1,0)\notdis R(a,0)\dis R(1-ba,b).
\end{equation*}
As $R(a,0)$ and $R(1-ba,b)$ are in distinct equivalence classes, so are
$R(1,0)$ and $R(1-ba,b)$. Therefore
\begin{equation*}
    \Mat2{1&0\\1-ba&b}\in\GL_2(R).
\end{equation*}
We deduce from Example \ref{bsp:matrizen} (\ref{bsp:matrizen.dreieck}) that $b$
is a unit. Thus, finally, $a=b^{-1}$ is a unit, too.

By the above, a product of two ring elements, with one factor in $J$, cannot be
a unit. Altogether, this means that $J$ is an ideal.
\end{proof}
\begin{nrtxt}\label{def:parallel}
If $R$ is a \emph{local ring\/} then two points $p,q\in\bP(R)$ are said to be
\emph{parallel\/}\index{points!parallel}\index{parallel points}, in symbols
$p\parallel q$, if they are non-distant. By the above this is an equivalence
relation and the equivalence classes of $\bP(R)$ are also called \emph{parallel
classes}. A definition of parallel points on the projective line over an
\emph{arbitrary ring\/} will be given in \ref{:parallel.allg}.

The following result is immediate from the proof of Theorem
\ref{thm:lokal.dist}:
\end{nrtxt}
\begin{cor}\label{cor:klassengroesse}
Let\/ $\bP(R)$ be the projective line over a local ring $R$. Then every
parallel class of\/ $\bP(R)$ has $\card \rad R$ elements.
\end{cor}
\begin{txt}
The relations ``$\parallel$'' and ``$=$'' coincide precisely when $R$ is a
field; see Theorem \ref{thm:koerper.dist}. In this case we get the finest
equivalence relation on $\bP(R)$, i.~e., parallel classes are singletons.

Our proof of Theorem \ref{thm:lokal.dist} could be shortened by using the
following characterisation of local rings (see \cite[Theorem~19.1]{lam-91}): A
ring $R$ is local if, and only if, $R\setminus R^*$ is a group under addition.
\end{txt}

\begin{nrtxt}
Suppose that $L$ is a field and that $K\subset L$ is a proper subfield
contained in the centre of $L$. Then the chain geometry $\Sigma(K,L)$ is called
a \emph{M\"{o}bius geometry}\index{Moebius geometry@M\"{o}bius geometry} in honour of
\textsc{August Ferdinand M\"{o}bius} (1790--1868). Two points of $\Sigma(K,L)$ are
distant precisely when they are distinct, since $L$ is a local ring and $\rad
R=\{0\}$. Hence there is a unique chain through any three distinct points.

Observe that the terminology in the literature is varying. We follow
\cite{herz-95} by assuming that $K$ is in the centre of $L$. Some authors drop
this condition and speak of a M\"{o}bius geometry $\Sigma(K,L)$ even if $K$ is just
a proper subfield of $L$. Also the term \emph{geometry of a field
extension\/}\index{geometry of a field extension}\index{field
extension!geometry of a} for such a chain geometry $\Sigma(K,L)$ is being used.
However, because of our emphasis on the finite case, this more general point of
view is irrelevant for our purposes. Cf.\ Theorem \ref{thm:endl.Moebius}.
\end{nrtxt}

\begin{exas}
Here are some examples of M\"{o}bius geometries and their generalisations. The
reader should consult \cite{benz-73} and \cite{blunck+he-05} for further
details.
\begin{enumerate}
\item

The classical example of a M\"{o}bius geometry is based on the fields $\bR$ and
$\bC=\bR\oplus \bR i$ of real and complex numbers. In fact,
$\Sigma(\bR,\bC)$ can be seen as an algebraic model of the geometry of
circles on a Euclidean $2$-sphere. There is a unique chain (circle) through
any three distinct points.

\item

Let $\bH=\bR\oplus \bR i\oplus \bR j\oplus\bR k$ denote the real quaternions.
Then $\Sigma(\bR,\bH)$ is a M\"{o}bius geometry which is isomorphic to the geometry
of circles on the Euclidean $4$-sphere. There is a unique chain (circle)
through any three distinct points.

\item

Another interesting classical example is $\Sigma(\bC,\bH)$, where $\bC$ is
identified with $\bR\oplus\bR i$. It is an algebraic model for the geometry of
$2$-spheres on a Euclidean $4$-sphere. Here there is more than one chain
through three distinct points. It is not a M\"{o}bius geometry according to our
definition, because the centre of the real quaternions is $\bR$.
\end{enumerate}

Now we turn to the finite case. Finite fields are commutative by a famous
theorem due to \textsc{Joseph Henry Mclagan-Wedderburn}\index{Wedderburn
theorem (on finite fields)}\index{theorem!of Wedderburn (on finite fields)}
(1882--1948) for which \textsc{Ernst Witt} (1911--1991) has given an elegant
short proof; cf.\ \cite{aigner+z-04}. Since finite commutative fields are
precisely the well known Galois fields, the finite M\"{o}bius geometries are easily
described.
\end{exas}

\begin{thm}\label{thm:endl.Moebius}
~
\begin{enumerate}

\item

Each finite M\"{o}bius geometry is of the form
$\Sigma\big({\GF(q)},\GF(q^h)\big)$, where $q\geq 2$ is a power of a prime
and $h\geq 2$ is an integer.

\item

Let $q\geq 2$ be a power of a prime and let $h\geq 1$ be an integer. Then
the chain geometry $\Sigma\big({\GF(q)},\GF(q^h)\big)$ is a $3$-design if
its chains are considered as ``blocks''. The parameters of this design are
\begin{equation*}
  v=q^h+1,\; k= q+1,\mbox{ and }\lambda_3=1.
\end{equation*}
\end{enumerate}
\end{thm}
\begin{proof}
(a) If $K$ is a proper subfield of a finite field $L$ then $K=\GF(q)$, where
$q\geq 2$ is a power of a prime, $L=\GF(q^h)$, and $h\geq 2$ equals the
dimension of $L$ over $K$, as a vector space\footnote{It is worth noting here
that $L=\GF(q^h)$ contains a unique subfield with $q$ elements.}.

(b) We have $\card \bP\big({\GF(q^h)}\big)=q^h+1$ according to
(\ref{eq:proj.gerade.lokal}). By their definition, all chains have $\card
\bP\big({\GF(q)}\big)=q+1$ elements. Since $L$ is commutative, every
multiplicative subgroup of $L^*$ is normal. Thus, by Theorem
\ref{thm:normalisator} (\ref{thm:normalisator.b}) applied to $K^*$ and $L^*$,
there is a unique chain through any three distinct points.
\end{proof}

\begin{txt}
In part (b) of the preceding Theorem we did not exclude the trivial case $h=1$,
even though it does not deserve our attention.
\end{txt}

\begin{nrtxt}\label{:laguerre.geom}
Suppose that $R$ is Laguerre algebra over $K$. Then $\Sigma(K,R)$ is called a
\emph{Laguerre geometry}\index{Laguerre geometry}. If, moreover, $R$ is finite
then the chain geometry $\Sigma(K,R)$ gives rise to a transversal divisible
$3$-design; it will be discussed in detail in Section \ref{sect:laguerre.DDs}.
\end{nrtxt}

\section{Notes and further references}

\begin{nrtxt}
There are several books and surveys on chain geometries and related concepts.
The publications \cite{benz-60}, \cite{benz-73}, \cite{benz+l+s-72},
\cite{benz+m-64}, \cite{benz+s+s-81}, \cite{blunck+he-05}, \cite{havl-07a}, and
\cite{herz-95} together with the references given there, cover these topics
from the very beginning up to the year 2006. Below we restrict our attention to
some recent publications.
\end{nrtxt}

\begin{nrtxt}
Various approaches have been made to axiomatise chain geometries, certain
classes of chain geometries, or structures sharing some properties with a
specific type of chain geometry.

This has lead to concepts like \emph{Benz planes}\index{Benz plane} (see
\cite[Section~5]{del-95}), \emph{weak chain spaces}\index{weak chain
space}\index{chain space!weak}, \emph{chain spaces}\index{chain space},
\emph{contact spaces}\index{contact space} (cf.\ \cite[Section~3]{herz-95},
\cite{oezc+h-99}), and \emph{circle planes}\index{circle plane}
\cite{blunck-01b}. However, in general those structures are much more general
than chain geometries. Nevertheless they can sometimes be described
algebraically in terms of a ring containing a subfield if some extra
assumptions are made. See \cite{blunck-97}, \cite{blunck-01d}, \cite{herz-98},
\cite{herz+k-96}, and \cite{herz+r-95}.

The investigation of \emph{topological circle planes}\index{circle
plane!topological}\index{topological circle plane} is part of the book
\cite{polster+s-01}. It contains a wealth of bibliographical data.
Characterisations of projective groups $\PGL_2(R)$, where $R$ is a ring, are
given in \cite{blunck-97a}, \cite{blunck-02a} and \cite{herz-96a}. See also
\cite[Chapter~6]{blunck+he-05}. Properties of projective lines over ``small''
rings are reviewed in \cite{saniga+p+k+p-07a} and \cite{saniga+p+p-06z}.

\end{nrtxt}

\begin{nrtxt}
On the other hand, it is possible to consider structures being more general
than associative algebras (e.g.\ \emph{alternative
algebras\/}\index{algebra!alternative} or \emph{Jordan systems\/}\index{Jordan
system}) in order to obtain a kind of ``chain geometry''. We refer to
\cite{bert+n-04a}, \cite{bert+n-05a}, \cite{blunck-92b}, \cite{blunck-94},
\cite{blunck-95}, \cite{blunck-96}, \cite{blunck-02a},
\cite[Chapter~3]{blunck+he-05}, \cite{blunck+s-95}, and \cite{herz-96b}.
\end{nrtxt}

\begin{nrtxt}
Other papers related with certain chain geometries are \cite{blunck-00a},
\cite{blunck-03a}, \cite{havl-93b}, \cite{havl-94b}, \cite{havl-97},
\cite{havl+l-03a}, \cite{herz-93a}, and \cite{herz-93b}. Every chain geometry
gives rise to \emph{partial affine spaces\/}\index{partial affine
space}\index{space!partial affine}. Such spaces are investigated in
\cite{herz+m-95}, \cite{meur-96}, and \cite{pamb-96}.
\end{nrtxt}


%% file: DD-Laguerre4.tex
\chapter{Divisible Designs via $\GL_2$-Actions}\label{chap:DD.GL}

\section{How to choose a base block}

\begin{nrtxt}
Let $R$ be a finite local ring. As before, we write ${\rad R}:=R\setminus R^*$
for its Jacobson radical. According to Theorem \ref{thm:lokal.dist} and by the
definition in \ref{def:parallel}, the relation ``parallel'' ($\parallel$) is an
equivalence relation on the projective line $\bP(R)$. Also, $\GL_2(R)$ is a
group acting on $\bP(R)$. In fact, we are in a position to apply Theorem
\ref{thm:spera}:
\end{nrtxt}

\begin{thm}\label{thm:spera.ring}
Let $R$ be a finite local ring, and let $B_0$ be a $\parallel$-transversal
subset of the projective line $\bP(R)$ with $k\geq 3$ points. Then
\begin{equation*}
    \big(\bP(R),\cB,{\parallel}\big)\mbox{~~with~~} \cB:=B_0^{\GL_2(R)}
\end{equation*}
is a $3$-$(s,k,\lambda_3)$-divisible design with $v=\card R +\card {\rad R}$
points, and $s=\card {\rad R}$.
\end{thm}

\begin{proof}
By Corollary \ref{cor:lokal.anzahl}, the projective line over $R$ has finite
cardinality $\card R + \card {\rad R}$. It was shown in Corollary
\ref{cor:klassengroesse} that all parallel-classes have $\card {\rad R}$
elements. According to its definition, the relation $\dis$ is a
$\GL_2(R)$-invariant notion. Recall that, by the definition in
\ref{def:parallel}, the relations $\parallel$ and $\notdis$ coincide for a
local ring. Therefore, also the equivalence relation $\parallel$ is
$\GL_2(R)$-invariant. Hence the assertion follows from Theorem \ref{thm:spera}.
\end{proof}

\begin{nrtxt}
While Theorem \ref{thm:spera.ring} shows that we can construct a wealth of DDs
from the projective line over a finite local ring, one essential problem
remains open:
\begin{center}
  \emph{What is the number of blocks containing a\/ $\parallel$-transversal $3$-set\/}?
\end{center}
Or, said differently:
\begin{center}
\emph{What is the value of the parameter $\lambda_3$}?
\end{center}
We read off from (\ref{eq:DD-act-lambda}) that to answer this question amounts
to finding two non-negative integers: Firstly, $\card \GL_2(R)$ and, secondly,
the cardinality of the setwise stabiliser of the base block $B_0$ under the
action of the general linear group $\GL_2(R)$. It is easy to determine the
order of the group $\GL_2(R)$; see the exercise below. However, it seems
impossible to state any result about the size of setwise stabiliser of $B_0$
without any further information concerning $B_0$.
\end{nrtxt}

\begin{exer}
Show that
\begin{equation}\label{eq:GL.GF}
  {\card\GL_2}\big({\GF(q)}\big)=(q^2-1)(q^2-q).
\end{equation}
Given a finite local ring $R$ with $R/\rad R\cong\GF(q)$ verify that
\begin{equation}\label{eq:GL.Ring}
  \card\GL_2(R)=(\card\rad R)^4 (q^2-1)(q^2-q).
\end{equation}
\end{exer}

\begin{nrtxt}
If $R$ is a finite local ring, but \emph{not\/} a local algebra (e.g.
$R=\bZ_4$), then the divisible designs which arise from $\bP(R)$ seem to be
unknown. We therefore have to exclude them from our discussion in the next
section.

It would be interesting learn more about the DDs which are based upon the
projective line over such a ring, for example the projective line over a
\emph{Galois ring}\index{Galois ring} \cite{wan-03a}. However, it seems to the
author as if there would not exist a ``natural'' choice for a base block.
\end{nrtxt}

\section{Transversal divisible designs from Laguerre algebras}\label{sect:laguerre.DDs}

\begin{nrtxt}
In applying Theorem \ref{thm:spera.ring}, we start with the easiest case, viz.\
the $3$-divisible designs defined by Laguerre geometries. Recall that for a
field $K$ which is contained in a ring $R$, as a subring, we write $\cC(K,R)$
for the set of $K$-chains of the projective line $\bP(R)$.

\end{nrtxt}

\begin{thm}\label{thm:dd.laguerre}
Let $R$ be an $h$-dimensional Laguerre algebra over $\GF(q)$, $1\leq h<\infty$.
Then
\begin{equation*}
    \big(\bP(R),\cC(\GF(q),R),{\parallel}\big)
\end{equation*}
is a transversal $3$-$(s,k,1)$-divisible design with $v=q^h+q^{h-1}$ points,
$s=q^{h-1}$, and $k=q+1$.
\end{thm}

\begin{proof}
The assertions on $v$ and $s$ follow immediately from Theorem
\ref{thm:spera.ring}, $\card R=q^h$, and $\card\rad R=q^{h-1}$. Also, we have
$k=\card\bP\big({\GF(q)}\big)=q+1=\frac{v}{s}$. Finally, since $\GF(q)$ is in
the centre of $R$, we obtain $\lambda_3=1$ by Example \ref{bsp:1kette}
(\ref{bsp:1kette.algebra}).
\end{proof}

\begin{txt}
As an immediate consequence we can show that there exist a lot of mutually
non-isomorphic transversal divisible designs:
\end{txt}

\begin{thm}\label{thm:dd.laguerre.exist}
Let $q\geq 2$ be a power of a prime and let $h\geq 1$ be a natural number. Then
there is at least one $h$-dimensional Laguerre algebra over $\GF(q)$. Therefore
at least one transversal $3$-$(s,k,1)$-DD with parameters as in Theorem
\noslant{\ref{thm:dd.laguerre}} exists.
\end{thm}

\begin{proof}
The assertion follows from Example \ref{bsp:lokale.ringe}
(\ref{bsp:lokale.ringe.reelle.hoeheredualzahlen}), by letting $K:=\GF(q)$.
\end{proof}

\begin{exer}
Determine the parameters $\lambda_2$, $\lambda_1$, and $\lambda_0$ (the number
of chains) of the DDs from Theorem \ref{thm:dd.laguerre}.
\end{exer}

\section{Divisible designs from local algebras}

\begin{nrtxt}
We shall frequently make use of the following result from algebra. It is known
as the \emph{Wedderburn principal theorem}\index{Wedderburn principal
theorem}\index{theorem!Wedderburn principal}:
\end{nrtxt}

\begin{thm}\label{thm:wedderburn}
Let $R$ be a finite local algebra over $K=\GF(q)$. Then there is a
$\GF(q)$-subalgebra $L$ of\/ $R$ which is isomorphic to the field $R/\rad R$
such that $R=\rad R\oplus L$.
\end{thm}

\begin{txt}
We refer to \cite[Theorem~VIII.28]{mcd-74} for a proof.
\end{txt}

\begin{nrtxt}
Given a finite-dimensional local algebra $R$ over $K=\GF(q)$ we have the
associated field $R/\rad R = \overline R$. The canonical epimorphism
$R\to\overline R$ takes $K$ to an isomorphic field which is a subring of
$\overline R$. So we obtain that
\begin{equation*}\label{}
 \overline R \cong \GF(q^m)\mbox{ for some natural number }m\geq 1.
\end{equation*}
This implies
\begin{equation*}
     \dim_K R = m + \dim_K(\rad R).
\end{equation*}
By the above and Theorem \ref{thm:wedderburn}, there is a field $L$ which is
isomorphic to $\overline R\cong\GF(q^m)$ such that $K\subset L\subset R$,
whence $R$ is a left vector space over $L$. We let
\begin{equation*}\label{}
    h:=\dim_L R\geq 1.
\end{equation*}
Hence
\begin{equation}\label{eq:teiler}
    \dim_K R=(\dim_L R)(\dim_K L)=hm
\end{equation}
and
\begin{equation}\label{eq:radikaldim}
    \dim_K (\rad R)=(h-1)m.
\end{equation}

The next theorem is taken from \cite[Example~2.5]{spera-92a}. It is a
generalisation of Theorem \ref{thm:dd.laguerre} which, of course, is included
as a particular case for $m=1$.
\end{nrtxt}

\begin{thm}\label{thm:dd.lokal}
Let $R$ be an finite-dimensional local algebra over $K=\GF(q)$, with $R / \rad
R\cong \GF(q^m)$, whence $\dim_K R= hm$ for some positive integer $h$. Then
\begin{equation*}
    \big(\bP(R),\cC(\GF(q),R),{\parallel}\big)
\end{equation*}
is a $3$-$(s,k,1)$-divisible design with $v=q^{hm}+q^{(h-1)m}$ points,
$s=q^{(h-1)m}$ and $k=q+1$.
\end{thm}

\begin{proof}
It suffices to repeat the proof of Theorem \ref{thm:dd.laguerre}, taking into
account that now $\card\rad R=q^{(h-1)m}$ by virtue of (\ref{eq:radikaldim}).
\end{proof}
\begin{txt}
Next, we apply this result to construct DDs:
\end{txt}

\begin{thm}\label{thm:dd.lokal.exist}
Let $q\geq 2$ be a power of a prime. Also, let $h$ and $m$ be a positive
integers. Then there is at least one $hm$-dimensional local algebra $R$ over
$\GF(q)$ with $R / \rad R\cong \GF(q^m)$. Therefore at least one
$3$-$(s,k,1)$-divisible design with parameters as in Theorem
\noslant{\ref{thm:dd.lokal}} exists.
\end{thm}

\begin{proof}
We infer from Theorem \ref{thm:dd.laguerre.exist} that there is an
$h$-dimensional Laguerre algebra $R$ over $\GF(q^m)$. Therefore $R/\rad R$ is
isomorphic to $\GF(q^m)$. This $R$ is an $hm$-dimensional local algebra over
$\GF(q)\subset \GF(q^m)$.
\end{proof}
\begin{txt}
Observe that for $h>1$ non-transversal DDs are obtained in this way.
\end{txt}

\begin{nrtxt}
By the definition of an (arbitrary) chain geometry $\Sigma(K,R)$, the group
$\GL_2(R)$ acts on $\bP(R)$ as a group of automorphisms of $\Sigma(K,R)$ or,
said differently, of the corresponding divisible design. Recall that
$\PGL_2(R)$ denotes the transformation group on $\bP(R)$ which is induced by
$\GL_2(R)$. However, in general this group is only a subgroup of the full
automorphism group.

We shall describe below the full automorphism group of certain chain geometries
and hence of the corresponding DDs. In order to do so we need the following
concept carrying the name of the German physicist \textsc{Pascual Jordan}
(1902--1980), who should not be confused with the French mathematician
\textsc{Camille Jordan} (1839--1922).
\end{nrtxt}

\begin{nrtxt}
Let $R$ and $R'$ be rings. A mapping $\sigma:R\to R'$ is called \emph{Jordan
homomorphism\/}\index{Jordan homomorphism}\index{rings!Jordan homomorphism of}
if
\begin{equation}\label{def:jordan}
  (a+b)^\sigma = a^\sigma + b^\sigma,\;\;
  1^\sigma = 1\;(\in R'),\;\;
  (aba)^\sigma = a^\sigma b^\sigma a^\sigma\;\;\;
  \mbox{for all } a,b\in R.
\end{equation}
See, among others, \cite[p.~2]{jac-68} or \cite[p.~832]{herz-95}. For such a
mapping $\sigma$ and any element $a\in R^*$ the equation
\begin{equation}\label{}
   1^\sigma=(aa^{-2}a)^\sigma = a^\sigma (a^{-2})^\sigma a^\sigma
\end{equation}
shows that $a^\sigma$ has a left and a right inverse, whence $a^\sigma$ is a
unit in ${R'}$. Also,
\begin{equation}\label{}
     a^\sigma=(aa^{-1}a)^\sigma=a^\sigma (a^{-1})^\sigma a^\sigma
\end{equation}
implies
\begin{equation}\label{}
      (a^{-1})^\sigma = (a^\sigma)^{-1} \mbox{ for all }a\in R^*.
\end{equation}
As usual, a bijective Jordan homomorphism is called a \emph{Jordan
isomorphism}\index{Jordan isomorphism of rings}\index{rings!Jordan isomorphism
of}; its inverse mapping is also a Jordan isomorphism.
\end{nrtxt}

\begin{nrtxt}
Let $\sigma:R\to R'$ be a mapping. If $\sigma$ is a homomorphism of rings then
it is also a Jordan homomorphism. This remains true if $\sigma:R\to R'$ is an
\emph{antihomomorphism}\index{antihomomorphism of
rings}\index{rings!antihomomorhism of}; this means that $\sigma$ is a
homomorphism of the additive groups, sends $1\in R$ to $1\in R'$, whereas
$(ab)^\sigma=b^\sigma a^\sigma$ for all $a,b\in R$. Of course, this
antihomomorphism $\sigma$ is at the same time a homomorphism if $R^\sigma$ is a
commutative subring of $R'$.

Let $\sigma:R\to R'$ be a Jordan homomorphism of rings. If $R$ and $R'$ are
commutative and if $1+1\in R^*$ then $\sigma$ is a homomorphism. If $R'$ has no
left or right zero divisors then $\sigma$ is a homomorphism or an
antihomomorphism. See, among others, \cite{bart+b-85}, \cite{herstein-56}, and
\cite[p.~114]{jac-85}. Thus under certain circumstances there will be no
\emph{proper\/}\index{proper Jordan homomorphism}\index{Jordan
homomorphism!proper} Jordan homomorphisms for two given rings, i.~e.\ Jordan
homomorphisms that are neither a homomorphism nor an antihomomorphism.
\end{nrtxt}

\begin{exas} We present some Jordan homomorphisms other than homomorphisms.
\begin{enumerate} \item A well known example of an antiautomorphism (a
bijective antihomomorphism of a ring onto itself) is as follows: Let $R$
commutative ring (or even a commutative field) and let $R^{m\times m}$ be the
ring of $m\times m$ matrices with entries from $R$ with $m\geq 2$. The
transposition of matrices is an antiautomorphism $R^{m\times m}\to R^{m\times
m}$.

\item Suppose that $R=\prod_{j\in J}R_j$ is the direct product of rings $R_j$.
Similarly, let $R'=\prod_{j\in J}R'_j$. Assume, furthermore, that $\sigma_j :
R_j \to R'_j$ is a family of mappings, where each $\sigma_j$ is a homomorphism
or an antihomomorphism. Then
\begin{equation*}
   \sigma:=\prod_{j\in J} \sigma_j : R \to R': (x_j)_{j\in J}\mapsto \left(x_j^{\sigma_j}\right)_{j\in J}
\end{equation*}
is Jordan homomorphism.

If among the mappings $\sigma_j$ there is a homomorphism, other than an
antihomomorphism, and an antihomomorphism, other than a homomorphism, then
$\sigma$ will be a proper Jordan homomorphism. Thus proper Jordan homomorphisms
can easily be found.

\item Let $\vV$ be a two-dimensional vector space over a commutative field $K$
and let $\vb_1,\vb_2$ be a basis. Then $(1,\vb_1,\vb_2,\vb_1\wedge\vb_2)$ is a
basis of the exterior algebra $\bigwedge\vV$; see \cite[Section~7.2]{jac-85}.
Hence there exists a unique $K$-linear bijection $\sigma:\bigwedge \vV\to
\bigwedge \vV$ with the following properties:  $\sigma$ interchanges $\vb_2$
with $\vb_1\wedge\vb_2$ and fixes the remaining basis elements $1$ and $\vb_1$.
In order to show that $\sigma$ is a Jordan isomorphism, it suffices to verify
the last condition in (\ref{def:jordan}) for the elements of the given basis.
As a matter of fact, that condition is satisfied in a trivial way: Clearly, it
is true if $a=1$ or $b=1$, otherwise it follows from
$\vv_1\wedge\vv_2\wedge\vv_3=0$ for all $\vv_1,\vv_2,\vv_3\in\vV$. Because of
\begin{equation*}
    (\vb_1\wedge\vb_2)^\sigma =\vb_2\neq 0,\mbox{~~and~~}
    \vb_1^\sigma\wedge \vb_2^\sigma=\vb_1\wedge\vb_1\wedge\vb_2=0,\;
\end{equation*}
the Jordan isomorphism $\sigma$ is proper.
\end{enumerate}
\end{exas}

\begin{nrtxt}
If a Jordan homomorphism of $K$-algebras is at the same time a $K$-linear
mapping then it is called a \emph{$K$-Jordan homomorphism}\index{K-Jordan
homomorphism@$K$-Jordan homomorphism}. The importance of $K$-Jordan
isomorphisms is illustrated by the following result, due to \textsc{Armin
Herzer}, which is presented without proof. See \cite[Theorem~9.2.1]{herz-95},
\cite{bart-89}, \cite{blunck+h-03}, and \cite[Chapter~4]{blunck+he-05} for
generalisations. Compare also with Proposition~2.3 and Proposition~3.6 in
\cite{herz-87a}.
\end{nrtxt}

\begin{thm}
Let $R$ and $R'$ be a local algebras over $K$. Then the following assertions
hold:
\begin{enumerate}
\item

If $\sigma: R\to R'$ is a $K$-Jordan isomorphism then the mapping
\begin{equation*}
  \bP(R)\to\bP(R'):\left\{
  \begin{array}{l}
  R(1,a)\mapsto R'(1^\sigma,a^\sigma),\\
  R(a,1)\mapsto R'(a^\sigma,1^\sigma),
  \end{array}
  \right.
\end{equation*}
is a well defined isomorphism of chain geometries.

\item If, moreover, $\card K\geq 3$ then every isomorphism of\/ $\Sigma(K,R)$
onto $\Sigma(K,R')$ is the product of a mapping as in \emph{(a)} and a
projectivity of\/ $\bP(R')$.
\end{enumerate}
\end{thm}

\begin{nrtxt}
By the above, we know not only all automorphisms of the DDs from Theorem
\ref{thm:dd.lokal}, but also all isomorphisms between such DDs, provided that
$\card K\geq 3$. Of course, ``to know'' means that the problem is reduced to
finding all $K$-Jordan isomorphisms between the underlying $K$-algebras.

According to \cite[Remark 4.3.2]{herz-87a}, there exist non-isomorphic Laguerre
algebras which give rise to isomorphic chain geometries and therefore, by
Theorem \ref{thm:dd.laguerre}, to isomorphic divisible designs. However, those
Laguerre algebras are Jordan isomorphic.
\end{nrtxt}

\section{Other kinds of blocks}

\begin{nrtxt}
The construction of a DD from a chain geometry over a finite local algebra, as
described in Theorem \ref{thm:dd.lokal}, can be generalised by modifying the
set of blocks as follows.
\end{nrtxt}

\begin{thm}\label{thm:dd.lokal.anders}
Let $R$ be an finite-dimensional local algebra over $K=\GF(q)$, with $R / \rad
R\cong \GF(q^m)$, whence $\dim_K R= hm$ for some positive integer $h$.
Furthermore, let $C_{0}$ be the standard chain of the chain geometry
$\Sigma(K,R)$, and suppose the base block $B_{0}$ to be chosen as follows:
\begin{enumerate}
\item $B_{0}:=C_{0}\setminus \{R(1,0)\}$, for $q>2$.

\item $B_{0}:=C_{0}\setminus \{R(1,0),R(0,1)\}$, for $q>3$.

\item $B_{0}:=C_{0}\setminus\{R(1,0),R(0,1),R(1,1)\}$, for $q>4$.
\end{enumerate}
This gives, according to Theorem \emph{\ref{thm:spera.ring}}, a
$3$-$(s,k,\lambda_3)$-divisible design with
\begin{equation*}
  v=q^{hm}+q^{(h-1)m} \mbox{~and~}s=q^{(h-1)m}.
\end{equation*}

The remaining parameters $k$ and $\lambda_3$ are
\begin{equation*}
\renewcommand{\arraystretch}{1.2}
\begin{array}{lll}
   k=q,  & \lambda_3=q-2,                             &\mbox{ in case \noslant{(a)}},\\
   k=q-1,& \lambda_3=\frac{1}{2}(q-2)(q-3),        &\mbox{ in case \noslant{(b)}},\\
   k=q-2,& \lambda_3=\frac{1}{6}(q-2)(q-3)(q-4), &\mbox{ in case \noslant{(c)}}.
\end{array}
\end{equation*}
\renewcommand{\arraystretch}{1.0}
\end{thm}

\begin{proof}
Firstly, we observe that $\card C_{0}=q+1$ and that $C_{0}$ is a
$\parallel$-transversal subset. So the assumptions on the cardinality of $q$
guarantee that $B_{0}$ has at least three points.

Next, since $\GL_2(R)$ acts $3$-$\dis$-transitively on $\bP(R)$, it suffices to
determine the number of blocks through $M:=\{R(1,0), R(0,1), R(1,1)\}$. By
Theorem \ref{thm:dd.laguerre}, the standard chain $C_{0}$ is the only chain
containing $M$. Henceforth any block containing $M$ has to be a subset of
$C_{0}$. There are $\SMat2{q-2\\j}$ possibilities to choose a $j$-set $W$ in
$C_{0}\setminus M$, where $j\in\{1,2,3\}$. We infer from Theorem
\ref{thm:3fach.kette} that each such $C_{0}\setminus W$ is a chain. This proves
the assertions on $\lambda_3$. The rest is clear from Theorem
\ref{thm:dd.lokal}.
\end{proof}

\begin{nrtxt}
The previous theorem is taken from \cite{giese+h+s-05a}. It suggests to remove
four or even more points from the standard chain in order to obtain a base
block for a $3$-DD. It is possible to treat the case for four points by
considering the number of \emph{cross ratios\/}\index{cross ratio} that arise
if those points are written in any order. In general, four distinct points
determine six cross ratios, but for a
\emph{harmonic}\index{tetrad!harmonic}\index{harmonic tetrad},
\emph{equianharmonic\/}\index{tetrad!equianharmonic}\index{equianharmonic
tetrad}, or
\emph{superharmonic\/}\index{tetrad!superharmonic}\index{superharmonic tetrad}
tetrad there are less than six values; cf.\ \cite[Section~6.1]{hirschfeld-98}.
Thus several cases have to be treated separately. We refer to
\cite{giese+h+s-05a}, and note that the results from there carry over
immediately to our slightly more general setting of a local algebra. Also, the
``complementary'' setting where a $4$-subset of the standard chain is chosen to
be the base block is described in \cite{giese+h+s-05a}. As before, cross ratios
are the key to calculating the parameter $\lambda_3$.
\end{nrtxt}

\begin{nrtxt}
Yet another ``natural choice'' of a base block is the projective line over such
a field $L\subset R$ which meets the requirements of the Wedderburn principle
theorem (see \ref{thm:wedderburn}). A general treatment of these DDs seems to
be missing in the literature. We present here the following example which is
based on \cite[Exercise~XIX.1]{mcd-74}. See also \cite{blunck+h+z-08a} for a
generalisation.
\end{nrtxt}

\begin{exa}
Let $L:=\GF(4)=\{0,1,\tau,\tau^2\}$ be the field with four elements. Its
multiplicative group is cyclic of order three. Addition in $L$ is subject to
$x+x=0$ for all $x\in L$, and $1+\tau=\tau^2$. The mapping $\sigma:L\to
L:x\mapsto x^2$ is easily seen to be an automorphism of order two.

We consider the local ring $R:=\GF(4)[\eps;\sigma]$ of twisted dual
numbers\index{numbers!twisted dual}\index{twisted dual numbers} over $L$. Thus
\begin{equation*}
  \eps^2=0\mbox{~and~}\eps x=x^\sigma\eps=x^2\eps\mbox{~for all~}x\in L;
\end{equation*}
cf.\ Example \ref{bsp:lokale.ringe} (\ref{bsp:lokale.ringe.twist}). $R$ is a
local algebra over $K:=\GF(2)\subset L$, but not an algebra over $L$, because
$\tau$ is not in the centre of $R$. The radical of $R$ is $\rad R =L\eps=\eps
L$. An isomorphism $R/\rad R\to L$ is given by $(x+y\eps)+\rad R\mapsto x$ for
all $x,y\in L$.

Following Theorem \ref{thm:normalisator} we determine the normaliser of $L^*$
in $R^*$. The units in $R^*$ have the form
\begin{equation*}
  n=x+y\eps\mbox{ with } x\in L^* \mbox{ and } y\in L.
\end{equation*}
Given such an $n$ we clearly have $n^{-1}1n=1$. By $n^{-1}\tau^2 n=(n^{-1}\tau
n)^2$, it remains to calculate $n^{-1}\tau n$. We obtain
\begin{eqnarray*}
  n^{-1}\tau n &=& (x+y\eps)^{-1}\tau (x+y\eps) \\
    &=& (x^{-1}-y\eps)\tau (x+y\eps)\\
   &=& \tau+x^{-1}\tau y\eps - y\eps \tau x - y\eps\tau y\eps \\
   &=& \tau+x^{-1}\tau y\eps - x^2y\tau^2 \eps  - y^3\tau^2 \eps^2 \\
   &=& \tau(1+x^2 y(1-\tau)\eps).
\end{eqnarray*}
As $x^2 y$ can assume all values in $L$, there are four possibilities, viz.\
\begin{equation*}
\begin{array}{lcl}
x^2y= 0 &:& n^{-1}\tau n =\tau\in L^*,\\
x^2y= 1 &:& n^{-1}\tau n =\tau+\eps\notin L^*,\\
x^2y=\tau &:& n^{-1}\tau n =\tau+\tau\eps\notin L^*,\\
x^2y=\tau^2 &:& n^{-1}\tau n =\tau+\tau^2\eps\notin L^*.
\end{array}
\end{equation*}
We infer that $n=x+y\eps$ is in the normaliser of $L^*$ in $R^*$ if, and only
if, $y=0$. Consequently, this normaliser coincides with $L^*$. By $\card L^*=3$
and $\card R^*=16-4=12$, there are four chains through any three mutually
distant points. Summing up, we have shown that
\begin{equation*}
    \big(\bP\big({\GF(4)[\eps;\sigma]}\big),\cC\big({\GF(4)},\GF(4)[\eps;\sigma]\big),{\parallel}\big)
\end{equation*}
is a transversal $3$-$(4,5,4)$-DD with $v=20$ points and $b=256$ blocks. As a
matter of fact, we actually have a $4$-$(4,5,1)$-DD: Given any
$\cR$-transversal $4$-set, say $\{p_0,p_1,p_2,p_3\}$, precisely one of the four
blocks through $p_0,p_1,p_2$ will contain $p_3$.
\end{exa}

\section{Notes and further references}

\begin{nrtxt}
All finite chain geometries (not only Laguerre geometries) have nice point
models in finite projective spaces, and models in terms of finite
Grassmannians. See the many references in \cite{blunck+h-00b},
\cite{blunck+h-00c}, \cite[Chapter~11]{blunck+he-05}, and
\cite[p.~812]{herz-95}. Thus, many DDs from this chapter allow---up to
isomorphism---other descriptions from which their connection with finite local
algebras may not be immediate.

For example, the DD which belongs to the algebra of dual numbers over $\GF(q)$
arises also as follows:
\begin{enumerate}
\item The points of the DD are the points of a quadratic cone without its
vertex in the three-dimensional projective space over $\GF(q)$. The blocks are
the non-degenerate conic sections of this cone. The point classes are the
generators of this cone, the vertex being removed from them. This is the finite
analogue of the Blaschke cone\index{Blaschke cone}.

\item The points of the DD are the lines of a parabolic linear congruence
without its axis in the three-dimensional projective space over $\GF(q)$. The
blocks are the reguli which are entirely contained in this congruence. The
point classes are the pencils of lines which are entirely contained in this
congruence, the axis being removed from them.
\end{enumerate}
The \emph{Klein mapping\/}\index{Klein mapping}---carrying the name of
\textsc{Felix Klein} (1849--1925)---is a one-one correspondence between the set
of lines of the three-dimensional projective space over a commutative field $K$
and the set of points of a certain quadric in a five-dimensional projective
space over $K$; it is called the \emph{Klein quadric}\index{Klein quadric}. A
reader who is familiar with this mapping will notice immediately that the Klein
image of the model in (b) is just the model described in (a). However, the
ambient space of the cone now is a three-dimensional tangent space of the Klein
quadric. Cf.\ \cite[15.4]{hirschfeld-85}.
\end{nrtxt}


%% file: DD-Laguerre5.tex
\chapter{An Outlook: Finite Chain Geometries}\label{chap:FCG}

\section{A parallelism based upon the Jacobson radical}

\begin{nrtxt}\label{:parallel.allg}
Now we turn our attention to the projective line over an arbitrary ring $R$, as
we present the announced definition of parallel points in the general case. It
is taken from \cite{blunck+h-03a}, where the term ``radical parallelism'' is
used instead: A point $p\in\bP(R)$ is called \emph{parallel\/}\index{parallel
points}\index{points!parallel} to a point $q\in\bP(R)$ if
\begin{equation*}
  x\dis p \,\Rightarrow\, x\dis q
\end{equation*}
holds for all $x\in\bP(R)$. In this case we write $p \parallel q$. By
definition, the distant relation on $\bP(R)$ is a $\GL_2(R)$-invariant notion.
Hence
\begin{equation}
  p\parallel q \,\Leftrightarrow\, p^\gamma \parallel q^\gamma
\end{equation}
holds for all $p,q\in\bP(R)$ and all $\gamma\in\GL_2(R)$.
\par
Clearly, the relation $\parallel$ is reflexive and transitive. We shall see
below that $\parallel$ is in fact an equivalence relation; also it will become
clear that our previous definition of parallel points ($R$ a local ring) is a
particular case of the definition from the above.
\end{nrtxt}

\begin{nrtxt}
The connection between the parallelism on $\bP(R)$ and the Jacobson radical of
$R$ (cf.\ \ref{:jacobson}) is as follows: We consider the factor ring $R/\rad
R=:\overline R$ and the canonical epimorphism $R\to \overline R:a\mapsto a+\rad
R=:\overline a$. It has the crucial property
\begin{equation}\label{eq:einheit-invariant}
  a\in R^* \,\Leftrightarrow\, \overline a\in \overline R\,^*
\end{equation}
for all $a\in R$; cf. \cite[Proposition~4.8]{lam-91}. The Jacobson radical of
the factor ring $R/\rad R$ is zero \cite[Proposition~4.6]{lam-91}.

In geometric terms we obtain a mapping
\begin{equation}\label{eq:def-phiquer}
  \bP(R)\to\bP(\overline R): p=
  R(a,b)\mapsto \overline R(\overline a,\overline b)=:\overline p
\end{equation}
which is well defined and surjective \cite[Proposition~3.5]{blunck+h-00b}.
Furthermore, as a geometric counterpart of (\ref{eq:einheit-invariant}) we have
\begin{equation}\label{eq:dist-invariant}
  p\dis q\,\Leftrightarrow\, \overline p\dis \overline q
\end{equation}
for all $p,q\in\bP(R)$, where we use the same symbol to denote the distant
relations on $\bP(R)$ and on $\bP(\overline R)$, respectively. See Propositions
3.1 and 3.2 in \cite{blunck+h-00b}. Of course, all this is a generalisation of
the mapping given in (\ref{eq:P(R)->P(Rquer)}), where $R$ was supposed to be
local.

The following is taken from Theorem~2.2 and Corollary~2.3 in
\cite{blunck+h-03a}:
\end{nrtxt}

\begin{thm}\label{thm:2}
The mapping given by \emph{(\ref{eq:def-phiquer})} has the property
\begin{equation}\label{eq:par=quergleich}
  p \parallel q \,\Leftrightarrow\, \overline p=\overline q
\end{equation}
for all $p,q\in\bP(R)$. Consequently, the parallelism \emph{($\parallel$)} on
the projective line over a ring is an equivalence relation.
\end{thm}

\begin{txt}
Let us write $[p]$ for the \emph{parallel class}\index{parallel class}\/ of
$p\in\bP(R)$. It can be derived from (\ref{eq:par=quergleich}) that
\begin{equation}\label{eq:gleichm}
 \card[p]=\card\rad R
\end{equation}
for all $p\in\bP(R)$. Thus the cardinality of $\rad R$ can be recovered from
the $\bP(R)$ as the cardinality of an arbitrarily chosen class of parallel
points. In particular, $\parallel$ is the equality relation if, and only if,
$\rad R=\{0\}$.
\par
An easy consequence of (\ref{eq:dist-invariant}) and Theorem \ref{thm:2} is
\begin{equation}\label{eq:par->nichtdist}
  p\parallel q
                 \,\Leftrightarrow\,
  \overline p=\overline q
                 \Rightarrow
  \overline p \,\notdis\, \overline q
                 \,\Leftrightarrow\,
  p \,\notdis\, q
\end{equation}
for all $p,q\in \bP(R)$. In general, however, the converse of
(\ref{eq:par->nichtdist}) is not true:
\end{txt}

\begin{thm}\label{thm:3}
Let $R$ be an arbitrary ring. The relations ``parallel'' \emph{($\parallel$)}
and ``non-distant'' \emph{($\notdis$)} on $\bP(R)$ coincide if, and only if,
$R$ is a local ring.
\end{thm}

For a proof we refer to \cite[Theorem~2.5]{blunck+h-03a}. By the above, our two
definitions of parallel points in \ref{def:parallel} and \ref{:parallel.allg}
coincide in case of a local ring.

\section{Counting the point set}

\begin{nrtxt}
Let $R$ be a finite ring. The problem to determine the number of points of the
projective line over $R$ is intricate. Our approach follows
\cite[Section~10]{veld-89} and it uses the following famous theorem on the
structure of semisimple rings due to \textsc{Joseph Henry Maclagan-Wedderburn}
and \textsc{Emil Artin}; cf.\ \cite[Theorem~3.5]{lam-91}. We state it only for
the particular case of a finite ring:
\end{nrtxt}

\begin{thm}\label{artin.wedderburn}\index{Wedderburn-Artin
theorem}\index{theorem!Wedderburn-Artin} Let $R$ be a finite ring such that
$\rad R$ is zero. Then $R$ is isomorphic to a direct product $R_1\times
R_2\times\cdots\times R_n$, where each $R_i$ is a full matrix ring
$\GF(q_i)^{m_i\times m_i}$. The number $n$ is uniquely determined, as are the
pairs $(m_i,q_i)$ for $i\in\{1,2,\ldots, n\}$.
\end{thm}

\begin{nrtxt}
It is possible to count the number of points of the projective line over the
ring of $m\times m$ matrices with entries from $\GF(q)$, because there exists a
bijection from this projective line onto the set of $m$-dimensional subspaces
of a $2m$-dimensional vector space over the same field. This result is due to
\textsc{Xavier Hubaut} \cite[p.~500]{hubaut-65}, who proved it for an arbitrary
commutative field $K$ instead of $\GF(q)$. This powerful tool was generalised
by \textsc{Andrea Blunck} \cite[Theorem~2.1]{blunck-99} to the ring of
endomorphisms of a vector space, without any restriction on its dimension or
the ground field. We add in passing that the projective line over a matrix ring
is essentially nothing else but a particular example of a \emph{projective
space of matrices}\index{projective space of matrices} as considered in
\cite[p.~124]{wan-96}; see also \cite{blunck+h-05a} and \cite{huanglp-06a}.

By virtue of this bijection and by a result of \textsc{Joseph Adolphe Thas}
\cite[3.3]{thas-71}, we obtain
\begin{equation}\label{eq:thas}
  \card\big(\bP(\GF(q)^{m\times m})\big)
  = \prod_{i=0}^{m-1}{\frac {{q}^{2m-i}-1}{{q}^{m-i}-1}}\,.
\end{equation}
See also \cite[Theorem~3.1]{hirschfeld-98}.

Next, it is easy to see that the projective line over a direct product of
rings, say
\begin{equation*}
  R_1\times R_2\times \cdots\times R_n,
\end{equation*}
is in one-one correspondence with the cartesian product\footnote{The case
$\bZ_6\cong \GF(2)\times \GF(3)$ is illustrated in Figure \ref{abb:Z6-dist}.}
\begin{equation*}
  \bP(R_1)\times \bP(R_2)\times \cdots \times \bP(R_n).
\end{equation*}
Hence the Wedderburn-Artin Theorem \ref{artin.wedderburn} and formula
(\ref{eq:thas}) provide the number of points on the projective line over a
direct product of matrix rings.

Finally, given any finite ring $R$ we infer from (\ref{eq:gleichm}) that
\begin{equation}\label{}
    \card\bP(R) = \big({\card\rad R}\big) \big({\card\bP(\overline R)}\big),
\end{equation}
where $\overline R=R/\rad R$. Since $\rad \overline R=0$, we can apply our
result from the above to count the number of points on $\bP(\overline R)$, thus
obtaining a formula for the number of points of the projective line $\bP(R)$.
\end{nrtxt}

\section{Divisible designs vs.\ finite chain geometries}

\begin{nrtxt}\label{:vergleich}
To end this series of lectures, let us compare the definition of a divisible
design from \ref{def:DD} with properties of a chain geometry $\Sigma(K,R)$,
where $R$ is a finite ring. Given $\Sigma(K,R)$ we can associate with it the
positive integers
\begin{equation}\label{eq:ketten.parameter}
    v:=\card\bP(R),\; t:=3,\; s_1:=\card\rad R,\; s_2:=v-\card R,\; k:=\card K+1,
    \mbox{ and } \lambda_t,
\end{equation}
where $\lambda_t$ is the constant number of blocks through any $t=3$ mutually
distant points. As we saw, $\lambda_t$ depends on ``how'' the field $K$ is
embedded in $R$, whence we cannot not state a precise value. We remark that
$v\geq \card R + \card\rad R$ implies the inequality
\begin{equation*}
  s_2\geq \card  R + \card \rad R-\card  R =\card \rad R.
\end{equation*}
\end{nrtxt}

\begin{nrtxt}
Given a finite chain geometry the following assertions hold, where we use the
constants introduced in (\ref{eq:ketten.parameter}):
\begin{itemize}

\item[(A$_1$)] $\card [x] =s_1$ for all $x \in \bP(R)$.

\item[(A$_2$)] $\card \{y\in\bP(R)\mid y\notdis x \}=s_2$ for all $x \in
\bP(R)$.

\item[(B$_1$)] $\cC(K,R)$ is a set of subsets of $\bP(R)$ with $\card C=k$ for
all chains $C \in\cC(K,R)$. The points of any chain are mutually distant.

\item[(C$_1$)] For each $t$-subset $Y\subset \bP(R)$ of mutually distant points
there exist a exactly $\lambda_t$ chains of $\mathcal{C}(K,R)$ containing $Y$.

\item[(D$_1$)] $t\leq \frac{v}{s_1}$.
\end{itemize}

Thus any finite chain geometry is ``almost'' a $3$-divisible design. However,
unless $R$ is a local ring, a $\parallel$-transversal $3$-subset of $\bP(R)$
need not be a subset of any chain, and the parameter $s_1$ need not coincide
with $s_2$.

On the other hand, the preceding conditions (A$_1$)--(D$_1$) could serve as a
starting point for the investigation of ``divisible design-like structures'' in
the future.
\end{nrtxt}
